\documentclass[a4paper,11pt]{amsart}

\usepackage{amsmath}
\usepackage{amssymb}
\usepackage{amsthm}
\usepackage[dvipdfmx]{graphicx}
\usepackage{amscd}
\usepackage{comment}
\usepackage{color}
\usepackage{setspace}    
\usepackage{url}                           

\begin{document}

\numberwithin{equation}{section}

\theoremstyle{plain}
\newtheorem{theorem}{Theorem}[section]
\newtheorem{proposition}[theorem]{Proposition}
\newtheorem{lemma}[theorem]{Lemma}
\newtheorem{corollary}[theorem]{Corollary}
\newtheorem{conjecture}{Conjecture}[section]

\theoremstyle{definition}
\newtheorem{definition}{Definition}[section]
\newtheorem{example}{Example}[section]
\newtheorem{problem}{Problem}[section]
\newtheorem{convention}{Convention}[section]

\newtheorem{remark}{Remark}[section]

\title[Acylindrical hyperbolicity of some Artin-Tits groups]{Acylindrical hyperbolicity of Artin-Tits groups associated to triangle-free graphs and cones over square-free bipartite graphs}

\author[M. Kato]{Motoko Kato}
\thanks{The first author is supported by JSPS KAKENHI Grant-in-Aid for Research Activity Start-up, Grant Number 19K23406 and JSPS KAKENHI Grant-in-Aid for Young Scientists, Grant Number 20K14311.}
\address[M. Kato]{Graduate School of Science and Engineering, Mathematics, Physics and Earth Sciences, Ehime University, 2-5 Bunkyo-cho, Matsuyama, Ehime, 790-8577 Japan}
\email{kato.motoko.yy@ehime-u.ac.jp}

\author[S. Oguni]{Shin-ichi Oguni}
\thanks{The second author is supported by JSPS KAKENHI Grant-in-Aid for Young Scientists (B), Grant Number 16K17595 and 20K03590.}
\address[S. Oguni]{Graduate School of Science and Engineering, Mathematics, Physics and Earth Sciences, Ehime University, 2-5 Bunkyo-cho, Matsuyama, Ehime, 790-8577 Japan}
\email{oguni.shinichi.mb@ehime-u.ac.jp}

\subjclass[]{}

\date{\today}         

\maketitle

\begin{abstract}
It is conjectured that the central quotient of any irreducible Artin-Tits group is either virtually cyclic or acylindrically hyperbolic.
We prove this conjecture for Artin-Tits groups that are known to be CAT(0) groups by a result of Brady and McCammond,
that is, Artin-Tits groups associated to graphs having no $3$-cycles and Artin-Tits groups of almost large type associated to graphs admitting appropriate directions.
In particular, the latter family contains Artin-Tits groups of large type associated to cones over square-free bipartite graphs. 
\end{abstract}

\section{Introduction}
Artin-Tits groups are groups with special finite presentations. 
Let $\Gamma$ be a finite simple graph
with the vertex set $V(\Gamma)$ and the edge set $E(\Gamma)$.
An edge $e$ consists of two endvertices, which we denote by $s_e$ and $t_e$.
We suppose that edges $e$ are labeled by integers $m(e)>1$.
The {\it Artin-Tits group} $A_{\Gamma}$ associated to $\Gamma$ is defined by the following presentation:
	\begin{align}\label{pres_standard_eq}
	A_{\Gamma}=\langle V(\Gamma)\mid \underbrace{s_e t_e s_e t_e \cdots}_{\text{length $m(e)$}}	=\underbrace{t_e s_e t_e s_e\cdots}_{\text{length $m(e)$}} \quad \text{for all $e\in E(\Gamma)$ } 	\rangle.
	\end{align}
Free abelian groups, free groups and braid groups are examples of Artin-Tits groups.
If we add relations $v^2=1$ to (\ref{pres_standard_eq}) for all $v\in V(\Gamma)$,
then we get the associated {\it Coxeter group} $W_{\Gamma}$.
In terms of the properties of $W_{\Gamma}$, we can define important classes of Artin-Tits groups.
For example, $A_{\Gamma}$ is said to be of {\it finite type} if $W_{\Gamma}$ is finite. Others are said to be of {\it infinite type}.
We mainly argue on Artin-Tits groups of infinite type.

We can also define classes of Artin-Tits groups in terms of edge labels of $\Gamma$.
$A_{\Gamma}$ is said to be
	\begin{itemize}
	\item {\it right-angled} if all edges of $\Gamma$ are labeled by $2$, or
	\item of {\it large type} if all edges of $\Gamma$ are labeled by integers greater than $2$.
	\end{itemize}

For general Artin-Tits groups, many basic questions are still open (refer to \cite{Charney2008PROBLEMSRT}).
For example, it is unknown whether the following are equivalent or not for any 
Artin-Tits group $A_{\Gamma}$ :
	\begin{itemize}
	\item[(i)] $A_{\Gamma}$ is directly indecomposable, that is, it does not decompose as a direct product of two non-trivial subgroups;
	\item[(ii)] $A_{\Gamma}$ is irreducible, that is, the defining graph $\Gamma$ does not decompose as a join of two non-empty subgraphs such that all edges between them are labeled by $2$.
	\end{itemize}
Note that (i) clearly implies (ii).

Now, we consider questions related to group actions on hyperbolic/non-positively curved spaces. 
The following is one of the biggest problems on such actions of Artin-Tits groups.
\begin{problem}[{\cite[Problem 4]{Charney2008PROBLEMSRT}}]\label{CAT(0)_conj}
Which Artin-Tits groups are {\it CAT(0) groups}, that is, groups acting geometrically on CAT(0) spaces?
\end{problem}
\noindent Here, {\it CAT(0) spaces} are geodesic spaces where every geodesic triangle is not fatter than the comparison triangle in the Euclidean plane (see \cite{MR1744486} for the precise definition).
A group action is said to be {\it geometric} if the action is proper, cocompact and by isometries.

The following is a related conjecture, which is the main concern of this paper.
\begin{conjecture}[{\cite[Conjecture B]{haettel2019xxl}}]\label{acyl_conj}
The central quotient of every irreducible Artin-Tits group is either virtually cyclic or acylindrically hyperbolic. 
\end{conjecture}

The definition of acylindrical hyperbolicity of groups is recalled in Section~\ref{triangle-free_sec}.
We can find many applications of acylindrically hyperbolic groups in \cite{MR3589159}, \cite{MR3430352} etc.
The $n$-strand braid group $B_n$ $(n\geq 3)$ is an Artin-Tits group and the central quotient is acylindrically hyperbolic
(\cite{MR1914565}, \cite{MR2390326}, \cite{MR2367021}).
In addition to this motivating example, Conjecture~\ref{acyl_conj} holds for Artin-Tits groups in the following list.
	\begin{itemize}
	\item[(A1)] Artin-Tits groups of finite type (\cite{MR3719080}).
	\item[(A2)] Right-angled Artin-Tits groups (\cite{MR2827012}, \cite{MR3192368}).
	\item[(A3)] Two-dimensional Artin-Tits groups such that the associated Coxeter groups are hyperbolic  (\cite{alex2019acylindrical}, see also \cite{haettel2019xxl} for $A_{\Gamma}$ such that all edges of $\Gamma$ are labeled by integers greater than $4$). Note that such a group is characterized as  $A_{\Gamma}$ such that every triplet $(v_1, v_2,v_3)$ of vertices of $\Gamma$ satisfies $$\frac{1}{m_{1,2}} + \frac{1}{m_{2,3}} + \frac{1}{m_{3,1}}<1,$$ where $$m_{i,j}=\begin{cases} \text{the label of the edge between $v_i$ and $v_j$} &(\text{if $v_i$ and $v_j$ are adjacent})\\\infty &(\text{if $v_i$ and $v_j$ are not adjacent}) \end{cases}$$ (\cite{Moussong}).
	 \item[(A4)] $A_{\Gamma}$ such that $\Gamma$ is not a join of two non-empty subgraphs (\cite{MR3993762}, see also \cite{MR3966610} for Artin-Tits groups of type FC such that the defining graphs have diameter greater than $2$).
	 \end{itemize}
	
To the above list, we add Artin-Tits groups that are known to be CAT(0) groups by a result of Brady and McCammond \cite{MR1770639}.
We discuss the following two cases.
The first case is when $\Gamma$ is triangle-free, that is, $\Gamma$ does not contain $3$-cycles. 
We do not need any restriction on the labels of edges.              
Such Artin-Tits groups are said to be {\it triangle-free Artin-Tits groups}.

\begin{theorem}\label{triangle-free_thm}
For every triangle-free Artin-Tits group $A_{\Gamma}$ such that $\Gamma$ has three or more vertices,
the following are equivalent.
\begin{itemize}
\item[(i)] $A_{\Gamma}$ is acylindrically hyperbolic.
\item[(ii)] $A_{\Gamma}$ is directly indecomposable.
\item[(iii)] $A_{\Gamma}$ is irreducible.
\item[(iv)] $\Gamma$ is not a complete bipartite graph with all edges labeled by $2$. 
\item[(v)] $\Gamma$ is disconnected, or it contains a $2$-path full subgraph with an edge labeled by an integer greater than $2$ or a $3$-path full subgraph with all edges labeled by  $2$.
\end{itemize}
Under either of (i)-(v), and thus all of (i)-(v),
$A_{\Gamma}$ is centerless. 
In particular, Conjecture~\ref{acyl_conj} is true for triangle-free Artin-Tits groups.
\end{theorem}

The second case is the following:
\begin{theorem}\label{generalizations_thm}
Let $A_{\Gamma}$ be an Artin-Tits group of almost large type associated to $\Gamma$ with three or more vertices.
Suppose that $\Gamma$ can be appropriately directed. 
Then the following are equivalent.
\begin{itemize}
\item[(i)] $A_{\Gamma}$ is acylindrically hyperbolic.
\item[(ii)] $A_{\Gamma}$ is directly indecomposable.
\item[(iii)] $A_{\Gamma}$ is irreducible.
\item[(iv)] $\Gamma$ is not a cone over a graph consisting of only isolated vertices with all edges labeled by $2$.
\end{itemize}
Under either of (i)-(iv), and thus all of (i)-(iv),
$A_{\Gamma}$ is centerless. 
In particular, Conjecture~\ref{acyl_conj} is true for Artin-Tits groups of almost large type associated to graphs admitting appropriate directions.
\end{theorem}

\noindent Terminologies ``of almost large type'' and ``appropriately directed'' are defined in Section~\ref{triangle-free_sec}.
We note that graphs can contain $3$-cycles in the setting of Theorem~\ref{generalizations_thm}.
As a corollary of Theorem~\ref{generalizations_thm}, we have the following.

\begin{corollary}\label{square-free_thm}
Let $A_{\Gamma}$ be an Artin-Tits group associated to a cone over a square-free bipartite graph
$\Gamma=\{v_0\}\ast \Gamma'$.
Suppose that $\Gamma$ has a $3$-cycle subgraph. 
Let $\Gamma'_1$ be the possibly empty subgraph consisting of all isolated vertices in $\Gamma'$, 
and let $\Gamma'_2=\Gamma'-\Gamma'_1$.
Suppose further that every edge in $\Gamma'_2$ and every edge between $v_0$ and $\Gamma'_2$ are labeled by integers greater than $2$.
Then $A_{\Gamma}$ is acylindrically hyperbolic, directly indecomposable and centerless. 
In particular, Conjecture~\ref{acyl_conj} is true for such an $A_{\Gamma}$. 
\end{corollary}
	\begin{figure}
	\begin{center}
	\hspace{0cm}
	\scalebox{0.6}{
	\input{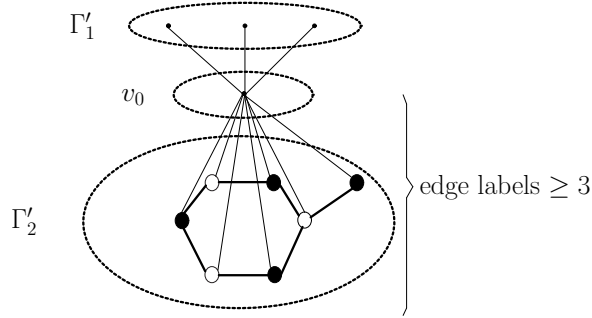}
	}
	\caption{The join of graphs $v_0$ and $\Gamma'_1\cup \Gamma'_2$.}\label{join_ex_fig}
	\end{center}
	\end{figure}
\noindent Figure~\ref{join_ex_fig} shows an example of $\Gamma$ in Corollary~\ref{square-free_thm}.

By Theorem~\ref{generalizations_thm}, we can also see that Conjecture~\ref{acyl_conj} is true for Artin-Tits groups of almost large type 
associated to square-free graphs (see Corollary~\ref{square-free_cor}).

\begin{remark}\label{trivial_rem}
Theorems~\ref{triangle-free_thm} and \ref{generalizations_thm} treated Artin-Tits groups $A_{\Gamma}$ such that $\Gamma$ has three or more vertices,
since Conjecture~\ref{acyl_conj} is true when $\Gamma$ has less than three vertices.
In fact, if $\Gamma$ has less than three vertices, it satisfies one of the following:
$(1)$ $\Gamma$ has only one vertex,
$(2)$ $\Gamma$ has exactly two vertices and no edges, and
$(3)$ $\Gamma$ has exactly two vertices and an edge.
In the first case, $A_{\Gamma}\cong \mathbb{Z}$, and thus the central quotient is trivial. 
In the second case, $A_{\Gamma}\cong \mathbb{F}_2$, which is hyperbolic. 
Since its center is trivial, the central quotient of $A_{\Gamma}$ is acylindrically hyperbolic.
In the third case, when the edge label is $2$, $A_{\Gamma}\cong \mathbb{Z}^2$, which is reducible.
When the edge label is greater than $2$, $A_{\Gamma}$ has an infinite cyclic center and its central quotient is $\mathbb{Z}/{m\mathbb{Z}}\ast \mathbb{Z}/2\mathbb{Z}$ for odd $m$ 
and $\mathbb{Z}/{\frac{m}{2}\mathbb{Z}}\ast \mathbb{Z}$ for even $m$ (see \cite{haettel2019xxl}).
These free products are virtually $\mathbb{F}_2$ and thus hyperbolic.
\end{remark}

We compare Artin-Tits groups in Theorem~\ref{triangle-free_thm} and Corollary~\ref{square-free_thm} with (A1)-(A4).
Our Artin-Tits groups are not necessarily of finite type, right-angled or with the associated Coxeter groups being hyperbolic.
A triangle-free Artin-Tits group is possibly associated to a join.
Also, a cone over a graph is a join. 
Figure~\ref{new_ex_fig} shows a triangle-free graph
and a cone over a square-free bipartite graph
such that associated Artin-Tits groups are not in (A1)-(A4).
	
	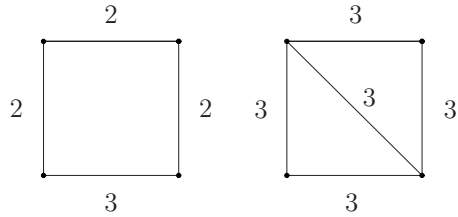
\begin{figure}
	\begin{center}
	\scalebox{0.7}{
{\unitlength 0.1in%
\begin{picture}(32.1000,14.0000)(3.1000,-15.3000)%
%
{\color[named]{Black}{%
\special{pn 8}%
\special{pa 3400 395}%
\special{pa 2400 395}%
\special{pa 2400 395}%
\special{pa 3400 395}%
\special{pa 3400 395}%
\special{pa 2400 395}%
\special{fp}%
}}%
%
{\color[named]{Black}{%
\special{pn 8}%
\special{pa 2400 395}%
\special{pa 2400 1395}%
\special{fp}%
\special{pa 2400 1395}%
\special{pa 3400 1395}%
\special{fp}%
\special{pa 3400 1395}%
\special{pa 3400 395}%
\special{fp}%
}}%
%
{\color[named]{Black}{%
\special{pn 8}%
\special{pa 3400 1395}%
\special{pa 2400 395}%
\special{fp}%
}}%
%
{\color[named]{Black}{%
\special{pn 4}%
\special{sh 1}%
\special{ar 3400 1395 16 16 0 6.2831853}%
\special{sh 1}%
\special{ar 2400 1395 16 16 0 6.2831853}%
\special{sh 1}%
\special{ar 2400 395 16 16 0 6.2831853}%
\special{sh 1}%
\special{ar 3400 395 16 16 0 6.2831853}%
\special{sh 1}%
\special{ar 3400 395 16 16 0 6.2831853}%
}}%
\put(29.1000,-1.9500){\makebox(0,0){{\color[named]{Black}{\Large $3$}}}}%
\put(36.1000,-9.0500){\makebox(0,0){{\color[named]{Black}{\Large $3$}}}}%
\put(30.1000,-8.1500){\makebox(0,0){{\color[named]{Black}{\Large $3$}}}}%
\put(22.1000,-9.0500){\makebox(0,0){{\color[named]{Black}{\Large $3$}}}}%
\put(28.8000,-15.9500){\makebox(0,0){{\color[named]{Black}{\Large $3$}}}}%
%
{\color[named]{Black}{%
\special{pn 8}%
\special{pa 1600 395}%
\special{pa 600 395}%
\special{pa 600 395}%
\special{pa 1600 395}%
\special{pa 1600 395}%
\special{pa 600 395}%
\special{fp}%
}}%
%
{\color[named]{Black}{%
\special{pn 8}%
\special{pa 600 395}%
\special{pa 600 1395}%
\special{fp}%
\special{pa 600 1395}%
\special{pa 1600 1395}%
\special{fp}%
\special{pa 1600 1395}%
\special{pa 1600 395}%
\special{fp}%
}}%
%
{\color[named]{Black}{%
\special{pn 4}%
\special{sh 1}%
\special{ar 1600 1395 16 16 0 6.2831853}%
\special{sh 1}%
\special{ar 600 1395 16 16 0 6.2831853}%
\special{sh 1}%
\special{ar 600 395 16 16 0 6.2831853}%
\special{sh 1}%
\special{ar 1600 395 16 16 0 6.2831853}%
\special{sh 1}%
\special{ar 1600 395 16 16 0 6.2831853}%
}}%
\put(11.0000,-1.9500){\makebox(0,0){{\color[named]{Black}{\Large $2$}}}}%
\put(11.0000,-15.9500){\makebox(0,0){{\color[named]{Black}{\Large $3$}}}}%
\put(4.0000,-8.9500){\makebox(0,0){{\color[named]{Black}{\Large $2$}}}}%
\put(18.0000,-8.9500){\makebox(0,0){{\color[named]{Black}{\Large $2$}}}}%
\end{picture}}%
	}
	\caption{Two examples of $\Gamma$ such that $A_{\Gamma}$ is not in (A1)-(A4).}\label{new_ex_fig}
	\end{center}
	\end{figure}

Finally, we consider one of the basic questions on algebraic properties of Artin-Tits groups.
When an irreducible Artin-Tits group is of finite type,
the center is known to be infinite cyclic (\cite{MR323910}, \cite{MR422673}).
For irreducible Artin-Tits groups of infinite type, it is conjectured that the center is trivial (\cite{Charney2008PROBLEMSRT}, \cite{MR3203644}).
When $\Gamma$ is not a cone, it is known that $A_{\Gamma}$ is centerless (\cite{MR3993762}).
Theorem~\ref{triangle-free_thm} and Theorem~\ref{generalizations_thm} give affirmative partial answers to
this conjecture.
In particular, Corollary~\ref{square-free_thm} claims that some Artin-Tits groups associated to cones are centerless.

We give an outline of this paper. 
Section~\ref{triangle-free_sec} contains preliminaries on acylindrically hyperbolic groups, Artin-Tits groups and Brady-McCammond's CAT(0) spaces.
Section~\ref{main_sec} contains proofs of Theorems~\ref{triangle-free_thm} and \ref{generalizations_thm}.
Our strategy is to answer the following problem: if an Artin-Tits group acts geometrically on a CAT(0) space, does it have a rank-one isometry?
Such a strategy is based on relations between rank one isometries on CAT(0) spaces and acylindrical hyperbolicity of groups (\cite{MR3415065}, \cite{MR3849623}), and was used in previous works on Artin-Tits groups (for example, \cite{haettel2019xxl}). 
In the proofs of the main theorems,
we observe geometric actions of Artin-Tits groups on CAT(0) spaces,
constructed by Brady and McCammond \cite{MR1770639}.
We detect group elements acting as rank one isometries on the CAT(0) spaces.

\section{Preliminaries}\label{triangle-free_sec}

\subsection{Acylindrically hyperbolic groups}
Hereafter, we always assume that group actions on metric spaces are by isometries.
We recall the definition of acylindrically hyperbolic groups.

\begin{definition}[\cite{MR2367021}, \cite{MR3430352}]
An isometric action of a group $G$ on a metric space $(X,d)$ is {\it acylindrical} if 
for every $\varepsilon\geq 0$, there exist $R\geq 0$ and $N\geq 0$ such that every $x$, $y\in X$ with $d(x,y)\geq R$ satisfy 
	\begin{align}
	|\{g\in G\mid d(x,gx)\leq \varepsilon, d(y,gy)\leq \varepsilon\}|\leq N.
	\end{align}
A group $G$ is {\it acylindrically hyperbolic} if $G$ acts acylindrically and non-elementarily on a (Gromov-) hyperbolic space.
\end{definition}
Examples and basic properties of acylindrically hyperbolic groups can be found in \cite{MR3430352}.

\begin{definition}[cf. \cite{MR1744486}]
An isometry $\gamma$ on a metric space $(X,d)$ is {\it hyperbolic} if there exists a point $x\in X$ satisfying 
$d(x,\gamma x)=\inf_{x'\in X}d(x',\gamma x') > 0$.
When $X$ is a CAT(0) space,
$\gamma$ is hyperbolic if and only if it acts by a translation on a geodesic line $l_{\gamma}$ in $X$.
We call $l_{\gamma}$ an {\it axis} of $\gamma$ (\cite[Theorem II-6.8]{MR1744486}).
$\gamma$ is {\it rank one} if it is hyperbolic and its axis does not bound a flat half plane.
\end{definition}

\begin{theorem}[{\cite[Theorem 1.3]{MR3849623}}]\label{acyl_thm}
If a group $G$ acts properly on a proper CAT(0) space with a rank one isometry, 
then $G$ is either virtually cyclic or acylindrically hyperbolic.
\end{theorem}

\subsection{Artin-Tits groups and Brady-McCammond's CAT(0) spaces}

Let $\Gamma$ be a finite simple graph with edges labeled by integers greater than $1$.
The associated Artin-Tits group $A_{\Gamma}$ is defined by the standard presentation (\ref{pres_standard_eq}).
A graph is said to be {\it directed} if every edge $e$ is identified with an ordered pair $(s_e,t_e)$ of endvertices.
When $\Gamma$ is directed, $A_{\Gamma}$ admits another presentation.

\begin{lemma}[{\cite[Section 5, Definition ($G_{\Gamma}$)]{MR1770639}}]\label{pres_gen_lem}
Let $\Gamma$ be a finite simple directed graph with 
the vertex set $V(\Gamma)$ and the edge set $E(\Gamma)$.
Suppose that edges $e$ are labeled by integers $m(e)>1$.
Then $A_{\Gamma}$ admits a presentation with the generating set
	\begin{align}\label{generators_eq}
	V(\Gamma)\cup \{x_e\}_{e\in E(\Gamma)}\cup \{\alpha_{e,3}, \ldots, \alpha_{e,m(e)}\}_{e\in E(\Gamma), m(e)\geq 3}
	\end{align}
and relations
	\begin{align}\label{pres_BM_eq}
	x_{e}=s_{e} t_{e}, x_{e}=t_{e} s_{e} 
	\end{align}
for every $e\in E(\Gamma)$ with $m(e)=2$, and 
	\begin{align}\label{pres_BM_eq_2}
	x_{e}=s_{e} t_{e}, x_{e} = t_{e}\alpha_{e,3}, \ldots, 
	x_{e}=\alpha_{e,i}\alpha_{e,i+1}, \ldots, x_{e}=\alpha_{e,m(e)}s_{e}
	\end{align}
for every $e\in E(\Gamma)$ with $m(e)\geq 3$.
\end{lemma}

Let $K_{\Gamma}$ be the presentation complex associated to the presentation of $A_{\Gamma}$ in Lemma~\ref{pres_gen_lem}.
$K_{\Gamma}$ has a unique vertex $o$, a directed 1-cell for each generator and a 2-cell for each relation in (\ref{pres_BM_eq}) and (\ref{pres_BM_eq_2}).
We denote by $\mathrm{Pr}:\widetilde{K}_{\Gamma}\to K_{\Gamma}$ the projection from the universal cover $\widetilde{K}_{\Gamma}$ onto $K_{\Gamma}$.
The $1$-skeleton of $\widetilde{K}_{\Gamma}$ can be identified with the Cayley graph of $A_{\Gamma}$ on the generators (\ref{generators_eq}).
We fix such an identification, 
and let $\widetilde{o}\in \widetilde{K}_{\Gamma}$ be the vertex corresponding to the identity element of $A_{\Gamma}$.
Figure~\ref{2-cells_fig} shows $2$-cells in the universal cover $\widetilde{K}_{\Gamma}$ of $K_{\Gamma}$.
	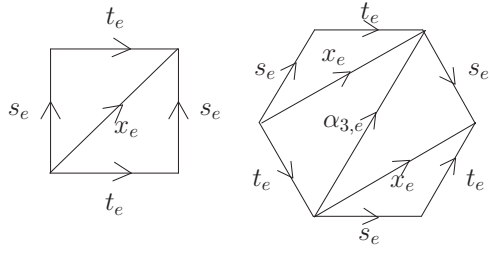
\begin{figure}
	\begin{center}
	\scalebox{0.7}{
{\unitlength 0.1in%
\begin{picture}(34.6400,16.7000)(3.2600,-17.7000)%
%
\special{pn 8}%
\special{pa 647 1363}%
\special{pa 647 437}%
\special{fp}%
\special{pa 647 437}%
\special{pa 1593 437}%
\special{fp}%
\special{pa 1593 437}%
\special{pa 1593 1363}%
\special{fp}%
\special{pa 1593 1363}%
\special{pa 647 1363}%
\special{fp}%
%
\special{pn 8}%
\special{pa 647 831}%
\special{pa 576 958}%
\special{fp}%
\special{pa 647 831}%
\special{pa 718 958}%
\special{fp}%
%
\special{pn 8}%
\special{pa 1593 831}%
\special{pa 1522 958}%
\special{fp}%
\special{pa 1593 831}%
\special{pa 1664 958}%
\special{fp}%
%
\special{pn 8}%
\special{pa 1244 1363}%
\special{pa 1120 1282}%
\special{fp}%
\special{pa 1244 1363}%
\special{pa 1108 1421}%
\special{fp}%
%
\special{pn 8}%
\special{pa 1238 437}%
\special{pa 1114 356}%
\special{fp}%
\special{pa 1238 437}%
\special{pa 1102 495}%
\special{fp}%
\put(11.2000,-15.9500){\makebox(0,0){\Large $t_e$}}%
\put(4.1100,-9.0000){\makebox(0,0){\Large $s_e$}}%
\put(18.2900,-9.0000){\makebox(0,0){\Large $s_e$}}%
\put(11.2600,-2.0500){\makebox(0,0){\Large $t_e$}}%
%
\special{pn 8}%
\special{pa 641 1363}%
\special{pa 1587 437}%
\special{fp}%
%
\special{pn 8}%
\special{pa 1180 855}%
\special{pa 1034 880}%
\special{fp}%
\special{pa 1180 855}%
\special{pa 1124 989}%
\special{fp}%
\put(12.0900,-10.3900){\makebox(0,0){\Large $x_e$}}%
%
\special{pn 8}%
\special{pa 3790 995}%
\special{pa 3400 295}%
\special{fp}%
%
\special{pn 8}%
\special{pa 3400 295}%
\special{pa 2600 295}%
\special{fp}%
%
\special{pn 8}%
\special{pa 2600 295}%
\special{pa 2190 985}%
\special{fp}%
%
\special{pn 8}%
\special{pa 2190 985}%
\special{pa 2590 1685}%
\special{fp}%
%
\special{pn 8}%
\special{pa 2590 1685}%
\special{pa 3390 1685}%
\special{fp}%
\special{pa 3390 1685}%
\special{pa 3790 985}%
\special{fp}%
%
\special{pn 8}%
\special{pa 2463 541}%
\special{pa 2350 596}%
\special{fp}%
\special{pa 2463 541}%
\special{pa 2449 665}%
\special{fp}%
%
\special{pn 8}%
\special{pa 3070 295}%
\special{pa 2965 225}%
\special{fp}%
\special{pa 3070 295}%
\special{pa 2955 345}%
\special{fp}%
%
\special{pn 8}%
\special{pa 3070 1695}%
\special{pa 2965 1625}%
\special{fp}%
\special{pa 3070 1695}%
\special{pa 2955 1745}%
\special{fp}%
%
\special{pn 8}%
\special{pa 3643 1251}%
\special{pa 3530 1306}%
\special{fp}%
\special{pa 3643 1251}%
\special{pa 3629 1375}%
\special{fp}%
%
\special{pn 8}%
\special{pa 3639 699}%
\special{pa 3526 644}%
\special{fp}%
\special{pa 3639 699}%
\special{pa 3625 575}%
\special{fp}%
%
\special{pn 8}%
\special{pa 2439 1399}%
\special{pa 2326 1344}%
\special{fp}%
\special{pa 2439 1399}%
\special{pa 2425 1275}%
\special{fp}%
%
\special{pn 8}%
\special{pa 2209 989}%
\special{pa 3419 299}%
\special{fp}%
%
\special{pn 8}%
\special{pa 3419 299}%
\special{pa 2599 1689}%
\special{fp}%
%
\special{pn 8}%
\special{pa 2599 1689}%
\special{pa 3789 989}%
\special{fp}%
%
\special{pn 8}%
\special{pa 2862 608}%
\special{pa 2736 616}%
\special{fp}%
\special{pa 2862 608}%
\special{pa 2802 717}%
\special{fp}%
%
\special{pn 8}%
\special{pa 3053 911}%
\special{pa 2940 966}%
\special{fp}%
\special{pa 3053 911}%
\special{pa 3039 1035}%
\special{fp}%
%
\special{pn 8}%
\special{pa 3302 1268}%
\special{pa 3176 1276}%
\special{fp}%
\special{pa 3302 1268}%
\special{pa 3242 1377}%
\special{fp}%
\put(27.4000,-4.9500){\makebox(0,0){\Large $x_e$}}%
\put(32.5000,-14.3500){\makebox(0,0){\Large $x_e$}}%
\put(22.3500,-6.0500){\makebox(0,0){\Large $s_e$}}%
\put(30.1000,-1.6500){\makebox(0,0){\Large $t_e$}}%
\put(30.0500,-18.3500){\makebox(0,0){\Large $s_e$}}%
\put(22.1000,-14.1500){\makebox(0,0){\Large $t_e$}}%
\put(37.9000,-14.1500){\makebox(0,0){\Large $t_e$}}%
\put(38.1500,-6.0500){\makebox(0,0){\Large $s_e$}}%
\put(28.1500,-10.0500){\makebox(0,0){\Large $\alpha_{3,e}$}}%
\end{picture}}%
	}
	\caption{$2$-cells in $\widetilde{K}_{\Gamma}$ corresponding to edges of $\Gamma$ labeled by $2$ and $3$.}\label{2-cells_fig}
	\end{center}
	\end{figure}
In \cite{MR1770639}, Brady and McCammond showed that 
$\widetilde{K}_{\Gamma}$ can be given a metric to be an $A_{\Gamma}$-equivariant CAT(0) space 
under some combinatorial assumptions on $\Gamma$.

Let us consider two families of Artin-Tits groups.
The first one is the family of triangle-free Artin-Tits groups.
\begin{theorem}[{\cite[Theorem 6]{MR1770639}}]\label{CAT(0)_thm}
Let $A_{\Gamma}$ be a triangle-free Artin-Tits group.
Let us assign $\Gamma$ an arbitrary direction.
Then $\widetilde{K}_{\Gamma}$ has a metric satisfying the following:
\begin{itemize}
\item all $2$-cells are isometric to a Euclidean isosceles right triangle, and
\item all $1$-cells corresponding to generators in $\{x_e\}_{e\in E(\Gamma)}$ are the longest and of length $3\sqrt{2}$.
\end{itemize}
Moreover, $\widetilde{K}_{\Gamma}$ with this metric is a proper CAT(0) space.
The action of $A_{\Gamma}$ on $\widetilde{K}_{\Gamma}$ is geometric.

	\begin{figure}
	\begin{center}
	\scalebox{0.7}{
	\input{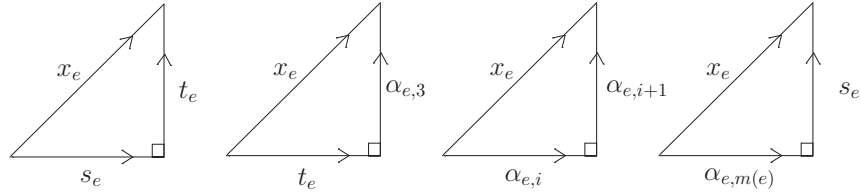}
	}
	\caption{$2$-cells of $\widetilde{K}_{\Gamma}$ isometric to an Euclidean isosceles right triangle.}\label{cells_isocseles_fig}
	\end{center}
	\end{figure}

\end{theorem}

Figure~\ref{cells_isocseles_fig} shows the identification of $2$-cells with Euclidean triangles in Theorem~\ref{CAT(0)_thm}.

The second one is the family of Artin-Tits groups of almost large type associated to graphs admitting appropriate directions. 
Here, an Artin-Tits group $A_{\Gamma}$ is said to be of {\it almost large type}
if the defining graph $\Gamma$ satisfies the following two conditions.
	\begin{itemize}
	\item[(1)] For every $3$-cycle in $\Gamma$, all edges are labeled by integers greater than $2$.
	\item[(2)] For every $4$-cycle in $\Gamma$, at least two edges are labeled by integers greater than $2$.
	\end{itemize}
In addition, we say that such a $\Gamma$ admits an {\it appropriate direction} or can be {\it appropriately directed} if $\Gamma$ can be directed such that each $3$- (resp.\ $4$-) cycle is directed in the same way as one of $3$- (resp.\ $4$-) cycles in Figure~\ref{squares2_fig}. 
We note that $\Gamma$ can admit $3$-cycles.

\begin{theorem}[{\cite[Theorem 7 and Remark on page 9]{MR1770639}}]\label{generalizations_CAT(0)_thm}
Let $A_\Gamma$ be an Artin-Tits group of almost large type and $\Gamma$ admit an appropriate direction.
Let us assign $\Gamma$ an appropriate direction.
Then $\widetilde{K}_{\Gamma}$ has a metric such that
all $2$-cells are isometric to a Euclidean equilateral triangle with side length $3$.
Moreover, $\widetilde{K}_{\Gamma}$ with this metric is a proper CAT(0) space.
The action of $A_{\Gamma}$ on $\widetilde{K}_{\Gamma}$ is geometric.
\end{theorem}
	
	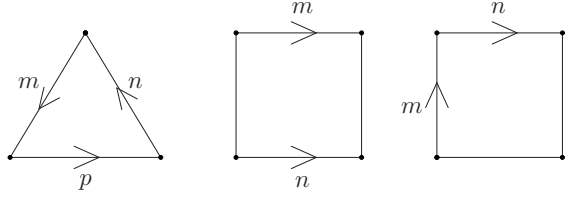
\begin{figure}
	\begin{center}
	\hspace{0cm}
	\scalebox{0.65}{
{\unitlength 0.1in%
\begin{picture}(44.0000,14.0000)(4.0000,-17.3500)%
%
{\color[named]{Black}{%
\special{pn 8}%
\special{pa 400 1600}%
\special{pa 1600 1600}%
\special{fp}%
\special{pa 1600 1600}%
\special{pa 1000 600}%
\special{fp}%
\special{pa 1000 600}%
\special{pa 400 1600}%
\special{fp}%
}}%
%
{\color[named]{Black}{%
\special{pn 8}%
\special{pa 2200 1600}%
\special{pa 2200 600}%
\special{fp}%
\special{pa 2200 600}%
\special{pa 3200 600}%
\special{fp}%
\special{pa 3200 600}%
\special{pa 3200 1600}%
\special{fp}%
\special{pa 3200 1600}%
\special{pa 2200 1600}%
\special{fp}%
}}%
%
{\color[named]{Black}{%
\special{pn 4}%
\special{sh 1}%
\special{ar 2200 1600 16 16 0 6.2831853}%
\special{sh 1}%
\special{ar 3200 1600 16 16 0 6.2831853}%
\special{sh 1}%
\special{ar 3200 600 16 16 0 6.2831853}%
\special{sh 1}%
\special{ar 2200 600 16 16 0 6.2831853}%
\special{sh 1}%
\special{ar 1600 1600 16 16 0 6.2831853}%
\special{sh 1}%
\special{ar 400 1600 16 16 0 6.2831853}%
\special{sh 1}%
\special{ar 1000 600 16 16 0 6.2831853}%
\special{sh 1}%
\special{ar 1000 600 16 16 0 6.2831853}%
\special{sh 1}%
\special{ar 1000 600 16 16 0 6.2831853}%
}}%
%
{\color[named]{Black}{%
\special{pn 8}%
\special{pa 1250 1042}%
\special{pa 1290 1216}%
\special{fp}%
\special{pa 1414 1131}%
\special{pa 1250 1042}%
\special{fp}%
}}%
%
{\color[named]{Black}{%
\special{pn 8}%
\special{pa 630 1210}%
\special{pa 670 1036}%
\special{fp}%
\special{pa 794 1121}%
\special{pa 630 1210}%
\special{fp}%
}}%
%
{\color[named]{Black}{%
\special{pn 8}%
\special{pa 1110 1600}%
\special{pa 910 1510}%
\special{fp}%
\special{pa 910 1690}%
\special{pa 1110 1600}%
\special{fp}%
}}%
%
{\color[named]{Black}{%
\special{pn 8}%
\special{pa 2840 1600}%
\special{pa 2640 1510}%
\special{fp}%
\special{pa 2640 1690}%
\special{pa 2840 1600}%
\special{fp}%
}}%
%
{\color[named]{Black}{%
\special{pn 8}%
\special{pa 2840 600}%
\special{pa 2640 510}%
\special{fp}%
\special{pa 2640 690}%
\special{pa 2840 600}%
\special{fp}%
}}%
%
{\color[named]{Black}{%
\special{pn 8}%
\special{pa 3800 1600}%
\special{pa 3800 600}%
\special{fp}%
\special{pa 3800 600}%
\special{pa 4800 600}%
\special{fp}%
\special{pa 4800 600}%
\special{pa 4800 1600}%
\special{fp}%
\special{pa 4800 1600}%
\special{pa 3800 1600}%
\special{fp}%
}}%
%
{\color[named]{Black}{%
\special{pn 8}%
\special{pa 3800 1000}%
\special{pa 3710 1200}%
\special{fp}%
\special{pa 3890 1200}%
\special{pa 3800 1000}%
\special{fp}%
}}%
%
{\color[named]{Black}{%
\special{pn 8}%
\special{pa 4440 600}%
\special{pa 4240 510}%
\special{fp}%
\special{pa 4240 690}%
\special{pa 4440 600}%
\special{fp}%
}}%
%
{\color[named]{Black}{%
\special{pn 4}%
\special{sh 1}%
\special{ar 3800 600 16 16 0 6.2831853}%
\special{sh 1}%
\special{ar 4800 600 16 16 0 6.2831853}%
\special{sh 1}%
\special{ar 4800 1600 16 16 0 6.2831853}%
\special{sh 1}%
\special{ar 3800 1600 16 16 0 6.2831853}%
\special{sh 1}%
\special{ar 3800 1600 16 16 0 6.2831853}%
}}%
\put(5.5000,-10.0000){\makebox(0,0){{\color[named]{Black}{\Large $m$}}}}%
\put(14.0000,-10.0000){\makebox(0,0){{\color[named]{Black}{\Large $n$}}}}%
\put(10.0000,-18.0000){\makebox(0,0){{\color[named]{Black}{\Large $p$}}}}%
\put(27.3000,-18.0000){\makebox(0,0){{\color[named]{Black}{\Large $n$}}}}%
\put(27.3000,-4.0000){\makebox(0,0){{\color[named]{Black}{\Large $m$}}}}%
\put(36.0000,-12.0000){\makebox(0,0){{\color[named]{Black}{\Large $m$}}}}%
\put(42.9000,-4.0000){\makebox(0,0){{\color[named]{Black}{\Large $n$}}}}%
\end{picture}}%
	}
	\caption{Directed $3$-cycles and $4$-cycles. Labels $m$, $n$, $p$ are greater than $2$. Unlabeled undirected edges can admit any label greater than $1$ and any direction.}\label{squares2_fig}
	\end{center}
	\end{figure}


In Theorems~\ref{CAT(0)_thm} and \ref{generalizations_CAT(0)_thm}, we assign $K_{\Gamma}$ a metric $d_{K_{\Gamma}}$
such that it is locally isometric to $\widetilde{K}_{\Gamma}$.
We often observe {\it the link} $L_{\Gamma}=\{z\in K_{\Gamma}\mid d_{K_{\Gamma}}(o,z)=1\}$ of the unique vertex $o$ in $K_{\Gamma}$.
Note that $L_{\Gamma}$ is regarded as a graph.
Indeed, each $1$-cell of $K_{\Gamma}$ corresponds to two vertices of $L_{\Gamma}$,
and each corner of a $2$-cell of $K_{\Gamma}$ corresponds to an edge of $L_{\Gamma}$.
We assign $L_{\Gamma}$ the path metric induced by the metric of $K_{\Gamma}$.
Then, the distance between adjacent vertices of $L_{\Gamma}$ is the angle between corresponding $1$-cells at $o$ in $K_{\Gamma}$.

We consider the setting of Theorem~\ref{CAT(0)_thm}.
For every directed $1$-cell $c$ of $K_{\Gamma}$, two intersection points with $L_{\Gamma}$ are named $c^-$ and $c^+$ in order, see Figure~\ref{cell_fig}.
We draw $L_{\Gamma}$ following \cite{MR1770639}, see Figure~\ref{link_fig}.
In $L_{\Gamma}$, every edge connected to a vertex $x_{e}^{+}$ or $x_{e}^{-}$ (a ``top'' or ``bottom'' edge in Figure~\ref{link_fig})
is of length $\pi/4$. 
The other edges (``middle'' edges in Figure~\ref{link_fig}) are of length $\pi/2$.
By noting that $\Gamma$ is triangle-free, we can confirm that $L_{\Gamma}$ does not contain non-trivial loops of length less than $2\pi$.
This fact is a key ingredient of the proof of Theorem~\ref{CAT(0)_thm} (\cite[Theorem 6]{MR1770639}).

Under the setting of Theorem~\ref{generalizations_CAT(0)_thm}, we can discuss everything in a similar way.
We note that all the edges of $L_{\Gamma}$ are of length $\pi/3$.

	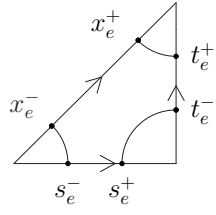
\begin{figure}
	\begin{center}
	\scalebox{0.7}{
{\unitlength 0.1in%
\begin{picture}(13.9000,13.3500)(2.7000,-15.3500)%
%
\special{pn 8}%
\special{pa 400 1400}%
\special{pa 1600 1400}%
\special{fp}%
\special{pa 1600 1400}%
\special{pa 1600 400}%
\special{fp}%
\special{pa 1600 400}%
\special{pa 1600 200}%
\special{fp}%
\special{pa 1600 200}%
\special{pa 400 1400}%
\special{fp}%
%
\special{pn 8}%
\special{pa 1150 1400}%
\special{pa 1040 1340}%
\special{fp}%
\special{pa 1040 1460}%
\special{pa 1150 1400}%
\special{fp}%
%
\special{pn 8}%
\special{pa 1600 820}%
\special{pa 1660 930}%
\special{fp}%
\special{pa 1540 930}%
\special{pa 1600 820}%
\special{fp}%
%
\special{pn 8}%
\special{pa 1050 755}%
\special{pa 930 790}%
\special{fp}%
\special{pa 1015 875}%
\special{pa 1050 755}%
\special{fp}%
\put(4.8000,-9.6000){\makebox(0,0){\Large $x_e^-$}}%
%
\special{pn 8}%
\special{ar 1600 200 400 400 1.5707963 2.3561945}%
%
\special{pn 8}%
\special{ar 1600 1400 400 400 3.1415927 4.7123890}%
%
\special{pn 8}%
\special{ar 400 1400 400 400 5.4977871 6.2831853}%
%
\special{pn 4}%
\special{sh 1}%
\special{ar 800 1400 16 16 0 6.2831853}%
\special{sh 1}%
\special{ar 680 1120 16 16 0 6.2831853}%
\special{sh 1}%
\special{ar 680 1120 16 16 0 6.2831853}%
%
\special{pn 4}%
\special{sh 1}%
\special{ar 1200 1400 16 16 0 6.2831853}%
\special{sh 1}%
\special{ar 1600 1000 16 16 0 6.2831853}%
\special{sh 1}%
\special{ar 1600 1000 16 16 0 6.2831853}%
%
\special{pn 4}%
\special{sh 1}%
\special{ar 1600 600 16 16 0 6.2831853}%
\special{sh 1}%
\special{ar 1320 480 16 16 0 6.2831853}%
\special{sh 1}%
\special{ar 1320 480 16 16 0 6.2831853}%
\put(10.8000,-3.6000){\makebox(0,0){\Large $x_e^+$}}%
\put(18.0000,-6.0000){\makebox(0,0){\Large $t_e^+$}}%
\put(18.0000,-10.0000){\makebox(0,0){\Large $t_e^-$}}%
\put(8.0000,-16.0000){\makebox(0,0){\Large $s_e^-$}}%
\put(12.0000,-16.0000){\makebox(0,0){\Large $s_e^+$}}%
\end{picture}}%
	}
	\caption{Intersection points of $1$-cells $x_e$, $s_e$ and $t_e$ with $L_{\Gamma}$.}
	\label{cell_fig}
	\end{center}
	\end{figure}
	\begin{figure}
	\begin{center}
	\scalebox{0.8}{
{\unitlength 0.1in%
\begin{picture}(29.5000,20.8000)(4.5000,-22.8000)%
\put(20.0000,-2.6500){\makebox(0,0){{\color[named]{Black}{\large $x_e^{+}$}}}}%
\put(20.1000,-23.4500){\makebox(0,0){{\color[named]{Black}{\large $x_e^{-}$}}}}%
\put(14.0000,-6.2500){\makebox(0,0){{\color[named]{Black}{\large $s_e^{+}$}}}}%
\put(14.0000,-20.2500){\makebox(0,0){{\color[named]{Black}{\large $s_e^{-}$}}}}%
%
{\color[named]{Black}{%
\special{pn 8}%
\special{pa 2010 400}%
\special{pa 2010 400}%
\special{fp}%
}}%
%
{\color[named]{Black}{%
\special{pn 8}%
\special{pa 2010 400}%
\special{pa 2610 800}%
\special{fp}%
\special{pa 2010 400}%
\special{pa 1410 800}%
\special{fp}%
}}%
%
{\color[named]{Black}{%
\special{pn 8}%
\special{pa 2210 800}%
\special{pa 2010 400}%
\special{fp}%
}}%
%
{\color[named]{Black}{%
\special{pn 8}%
\special{pa 2010 400}%
\special{pa 1810 800}%
\special{fp}%
}}%
%
{\color[named]{Black}{%
\special{pn 8}%
\special{pa 2010 2180}%
\special{pa 2010 2180}%
\special{fp}%
}}%
%
{\color[named]{Black}{%
\special{pn 8}%
\special{pa 2010 2180}%
\special{pa 2610 1780}%
\special{fp}%
\special{pa 2010 2180}%
\special{pa 1410 1780}%
\special{fp}%
}}%
%
{\color[named]{Black}{%
\special{pn 8}%
\special{pa 2210 1780}%
\special{pa 2010 2180}%
\special{fp}%
}}%
%
{\color[named]{Black}{%
\special{pn 8}%
\special{pa 2010 2180}%
\special{pa 1810 1780}%
\special{fp}%
}}%
%
{\color[named]{Black}{%
\special{pn 8}%
\special{pa 1400 800}%
\special{pa 2600 1800}%
\special{fp}%
}}%
%
{\color[named]{Black}{%
\special{pn 8}%
\special{pa 2600 800}%
\special{pa 2200 1800}%
\special{fp}%
}}%
%
{\color[named]{Black}{%
\special{pn 8}%
\special{pa 1800 1800}%
\special{pa 2200 800}%
\special{fp}%
}}%
%
{\color[named]{Black}{%
\special{pn 8}%
\special{pa 1800 800}%
\special{pa 1400 1800}%
\special{fp}%
}}%
\put(21.000,-16.6000){\makebox(0,0){{\color[named]{Black}{\large $\alpha_{e,3}^{-}$}}}}%
\put(17.000,-16.6000){\makebox(0,0){{\color[named]{Black}{\large $\alpha_{e,4}^{-}$}}}}%
\put(19.2000,-9.2000){\makebox(0,0){{\color[named]{Black}{\large $\alpha_{e,4}^{+}$}}}}%
\put(23.2000,-9.2000){\makebox(0,0){{\color[named]{Black}{\large $\alpha_{e,3}^{+}$}}}}%
\put(26.0000,-6.2000){\makebox(0,0){{\color[named]{Black}{\large $t_e^{+}$}}}}%
\put(26.0000,-20.2000){\makebox(0,0){{\color[named]{Black}{\large $t_e^{-}$}}}}%
%
{\color[named]{Black}{%
\special{pn 4}%
\special{sh 1}%
\special{ar 2600 1800 16 16 0 6.2831853}%
\special{sh 1}%
\special{ar 2200 1800 16 16 0 6.2831853}%
\special{sh 1}%
\special{ar 1800 1800 16 16 0 6.2831853}%
\special{sh 1}%
\special{ar 1400 1800 16 16 0 6.2831853}%
\special{sh 1}%
\special{ar 2000 2180 16 16 0 6.2831853}%
\special{sh 1}%
\special{ar 2600 800 16 16 0 6.2831853}%
\special{sh 1}%
\special{ar 2200 800 16 16 0 6.2831853}%
\special{sh 1}%
\special{ar 1800 800 16 16 0 6.2831853}%
\special{sh 1}%
\special{ar 1400 800 16 16 0 6.2831853}%
\special{sh 1}%
\special{ar 2000 400 16 16 0 6.2831853}%
\special{sh 1}%
\special{ar 2000 400 16 16 0 6.2831853}%
}}%
%
{\color[named]{Black}{%
\special{pn 8}%
\special{pa 600 800}%
\special{pa 3400 800}%
\special{dt 0.045}%
}}%
%
{\color[named]{Black}{%
\special{pn 8}%
\special{pa 600 1800}%
\special{pa 3400 1800}%
\special{dt 0.045}%
}}%
\put(6.0000,-2.7000){\makebox(0,0){{\color[named]{Black}{\large Edge length}}}}%
\put(6.0000,-5.7000){\makebox(0,0){{\color[named]{Black}{\Large $\frac{\pi}{4}$}}}}%
\put(6.0000,-13.7000){\makebox(0,0){{\color[named]{Black}{\Large $\frac{\pi}{2}$}}}}%
\put(6.0000,-21.7000){\makebox(0,0){{\color[named]{Black}{\Large $\frac{\pi}{4}$}}}}%
\end{picture}}%
	}
	\caption{A part of $L_{\Gamma}$ related to an edge $e$ labeled by $4$ in Theorem~\ref{CAT(0)_thm}.}\label{link_fig}
	\end{center}
	\end{figure}
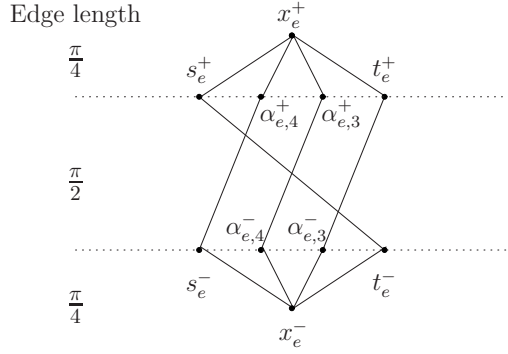

\section{Main results}\label{main_sec}
\subsection{Triangle-free Artin-Tits groups}
In this section, we prove Theorem~\ref{triangle-free_thm}.

\begin{lemma}\label{key1_lem}
Let $A_{\Gamma}$ be a triangle-free Artin-Tits group
and let $\Gamma$ be assigned an arbitrary direction, as in Theorem~\ref{CAT(0)_thm}.
Suppose further that $\Gamma$ contains one of the following directed graphs as a full subgraph:
	\begin{itemize}
	\item[(1)] the $2$-path directed graph $\Gamma_{m,n}$ as in Figure~\ref{key_fig}, and
	\item[(2)] the $3$-path directed graph $\Gamma_{2,2,2}$ as in Figure~\ref{key_fig}.
	\end{itemize}
Then $A_{\Gamma}$ is acylindrically hyperbolic.

	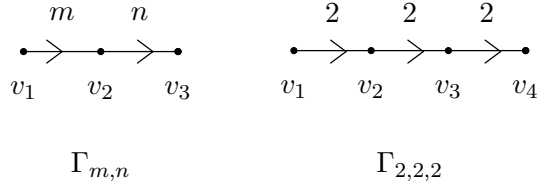
\begin{figure}
	\begin{center}
	\scalebox{1.0}{
{\unitlength 0.1in%
\begin{picture}(27.5000,8.0500)(4.5000,-11.3500)%
%
\special{pn 4}%
\special{sh 1}%
\special{ar 600 600 16 16 0 6.2831853}%
\special{sh 1}%
\special{ar 1000 600 16 16 0 6.2831853}%
\special{sh 1}%
\special{ar 1400 600 16 16 0 6.2831853}%
\special{sh 1}%
\special{ar 1400 600 16 16 0 6.2831853}%
\put(8.0000,-4.0000){\makebox(0,0){$m$}}%
\put(12.0000,-4.0000){\makebox(0,0){$n$}}%
%
\special{pn 4}%
\special{sh 1}%
\special{ar 2000 595 16 16 0 6.2831853}%
\special{sh 1}%
\special{ar 2400 595 16 16 0 6.2831853}%
\special{sh 1}%
\special{ar 2800 595 16 16 0 6.2831853}%
\special{sh 1}%
\special{ar 3200 595 16 16 0 6.2831853}%
\special{sh 1}%
\special{ar 3200 595 16 16 0 6.2831853}%
\put(22.0000,-3.9500){\makebox(0,0){$2$}}%
\put(26.0000,-3.9500){\makebox(0,0){$2$}}%
\put(30.0000,-3.9500){\makebox(0,0){$2$}}%
%
\special{pn 8}%
\special{pa 2000 595}%
\special{pa 3200 595}%
\special{fp}%
%
\special{pn 8}%
\special{pa 600 600}%
\special{pa 1400 600}%
\special{fp}%
\put(6.0000,-8.0000){\makebox(0,0){$v_1$}}%
\put(10.0000,-8.0000){\makebox(0,0){$v_2$}}%
\put(14.0000,-8.0000){\makebox(0,0){$v_3$}}%
\put(20.0000,-7.9500){\makebox(0,0){$v_1$}}%
\put(24.0000,-7.9500){\makebox(0,0){$v_2$}}%
\put(28.0000,-7.9500){\makebox(0,0){$v_3$}}%
\put(32.0000,-7.9500){\makebox(0,0){$v_4$}}%
\put(10.0000,-12.0000){\makebox(0,0){$\Gamma_{m,n}$}}%
\put(26.0000,-12.0000){\makebox(0,0){$\Gamma_{2,2,2}$}}%
%
\special{pn 8}%
\special{pa 800 600}%
\special{pa 720 540}%
\special{fp}%
\special{pa 720 660}%
\special{pa 800 600}%
\special{fp}%
%
\special{pn 8}%
\special{pa 1270 600}%
\special{pa 1190 540}%
\special{fp}%
\special{pa 1190 660}%
\special{pa 1270 600}%
\special{fp}%
%
\special{pn 8}%
\special{pa 2270 600}%
\special{pa 2190 540}%
\special{fp}%
\special{pa 2190 660}%
\special{pa 2270 600}%
\special{fp}%
%
\special{pn 8}%
\special{pa 2670 600}%
\special{pa 2590 540}%
\special{fp}%
\special{pa 2590 660}%
\special{pa 2670 600}%
\special{fp}%
%
\special{pn 8}%
\special{pa 3070 600}%
\special{pa 2990 540}%
\special{fp}%
\special{pa 2990 660}%
\special{pa 3070 600}%
\special{fp}%
\end{picture}}%
	}
	\vspace{0.2cm}
	\caption{A $2$-path directed graph $\Gamma_{m,n}$ with edges labeled by $m\geq 2$ and $n\geq 3$, and a $3$-path directed graph $\Gamma_{2,2,2}$ with edges labeled by $2$.}\label{key_fig}
	\end{center}
	\end{figure}
\end{lemma}

\begin{proof}

Let $\widetilde{K}_{\Gamma}$ be the $A_{\Gamma}$-equivariant CAT(0) space in Theorem~\ref{CAT(0)_thm}.
Let $L_{\Gamma}$ be the link of the unique vertex $o$ of $K_{\Gamma}$.
We find a rank one isometry in $A_{\Gamma}$.

$(1)$ 
We discuss the case where $\Gamma$ contains $\Gamma_{m,n}$ as a full subgraph.
By Lemma~\ref{pres_gen_lem}, $A_{\Gamma_{m,n}}$ has the following presentation:
        \begin{align}
	\left\langle 
	\begin{array}{l|l}
	v_1, v_2, v_3, x_1, x_2, &x_1=v_1v_2, x_1=v_2\alpha_{1,3}, \ldots, x_1=\alpha_{1,m} v_1,\\
	\alpha_{1,3}, \ldots, \alpha_{1,m}, &x_2=v_2v_3, x_2=v_3\alpha_{2,3}, \ldots, x_2=\alpha_{2,n} v_2\\
	\alpha_{2,3}, \ldots, \alpha_{2,n}
	\end{array}
	\right\rangle. \label{pres_m_n_eq}
	\end{align}

Let $K_{\Gamma_{m,n}}\subset K_{\Gamma}$ be the presentation complex of (\ref{pres_m_n_eq}).
Let $L_{\Gamma_{m,n}}\subset L_{\Gamma}$ be the link in $K_{\Gamma_{m,n}}$.
We draw $L_{\Gamma_{m,n}}$ on the lefthandside of Figure~\ref{link_2_3_fig}.

We show that $\alpha_{2,n}v_3v_1$ acts as a rank one isometry on $\widetilde{K}_{\Gamma}$.
First we find an axis of $\alpha_{2,n}v_3v_1$.
In $K_{\Gamma}$, 
let $l$ be the concatenation of $1$-cells $\alpha_{2,n}$, $v_3$ and $v_1$ in this order.
We note that $1$-cells of $K_{\Gamma}$ are loops based at $o$. 
We show that $l$ is a local geodesic.
Since $l$ is geodesic around any point of $l$ except $o$,
we investigate $l$ around $o$ and show that 
	\begin{align}\label{angle-results_eq}
	d_{L_{\Gamma}}(v_1^{+}, \alpha_{2,n}^{-}),\ 
	d_{L_{\Gamma}}(\alpha_{2,n}^{+}, v_3^-),\ 
	d_{L_{\Gamma}}(v_3^{+}, v_1^{-}) \geq \pi.
	\end{align}
To see $d_{L_{\Gamma}}(v_1^{+}, \alpha_{2,n}^{-})\geq \pi$,
we find a shortest path $P$ from $v_1^{+}$ to $\alpha_{2,n}^{-}$ uniquely in $L_{\Gamma_{m,n}}$, which is of length $\pi$ (see the bold line in Figure~\ref{link_2_3_fig}).
This path $P$ is the unique shortest one even in $L_{\Gamma}$.
Indeed, if we have a shortest path $P'$ from $v_1^{+}$ to $\alpha_{2,n}^{-}$ in $L_{\Gamma}$ through $L_{\Gamma}-L_{\Gamma_{m,n}}$,
then $P'$ should go out from $L_{\Gamma_{m,n}}$ at a point in $\{v_i^{\pm}\}_{i=1,2,3}$
and come back into $L_{\Gamma_{m,n}}$ at another point in $\{v_i^{\pm}\}_{i=1,2,3}$.
Since $\Gamma$ is simple and $\Gamma_{m,n}$ is a full subgraph, every path through $\Gamma-\Gamma_{m,n}$ in $\Gamma$ between different vertices in $\{v_i\}_{i=1,2,3}$ contains at least two edges.
It follows that every path through $L_{\Gamma}- L_{\Gamma_{m,n}}$ between different points of $\{v_i^{\pm}\}_{i=1,2,3}$ is of length greater than or equal to $\pi$, see Figure~\ref{2links_fig}.
This shows that the length of $P'$ is greater than $\pi$.
In particular we have
	\begin{align}\label{angle-results1_eq}
	d_{L_{\Gamma}}(v_1^{+}, \alpha_{2,n}^{-}) = d_{L_{\Gamma_{m,n}}}(v_1^{+}, \alpha_{2,n}^{-})=\pi. 
	\end{align}
Similarly, we confirm that 
	\begin{align}\label{angle-results2_eq}
	d_{L_{\Gamma}}(\alpha_{2,n}^{+}, v_3^-) = d_{L_{\Gamma_{m,n}}}(\alpha_{2,n}^{+}, v_3^-)=\pi.
	\end{align}
Also, according to Figure~\ref{link_2_3_fig}, $d_{L_{\Gamma_{m,n}}}(v_3^{+}, v_1^{-}) > \pi$.
Therefore, 
	\begin{align}\label{angle-results3_eq}
	d_{L_{\Gamma}}(v_3^{+}, v_1^{-}) \geq  \pi.
	\end{align}
Here, the minimum $\pi$ occurs only when $v_1$ and $v_3$ are connected by a $2$-path $\Gamma'$ in Figure~\ref{2links_fig}.
By (\ref{angle-results1_eq}),   (\ref{angle-results2_eq}) and  (\ref{angle-results3_eq}),
 we have (\ref{angle-results_eq}).
It follows that $l$ is geodesic around $o$.
Indeed, assume that a geodesic from $v_1^+$ to $\alpha_{2,n}^-$,
a geodesic from $\alpha_{2,n}^+$ to $v_3^-$ or a geodesic from $v_3^+$ to $v_1^-$
does not pass through $o$.
Then a triplet $(v_1^+, o, \alpha_{2,n}^-)$, $(\alpha_{2,n}^+, o, v_3^-)$ or $(v_3^+, o, v_1^-)$
contributes to a non-collapsing geodesic triangle as vertices.
Since $K_\Gamma$ is piecewise Euclidean and CAT(0) around $o$,
three interior angles of such a triangle must be less than $\pi$, contrary to (3.2)
(cf.\ discussions in the proof of Theorem 15 of \cite{MR1086658}).
Therefore $l$ is locally geodesic.
Hence the lift $\widetilde{l}$ through $\widetilde{o}$ is an axis of $\alpha_{2,n}v_3v_1$, see Figure~\ref{l_2_3_fig}.

Next we show that $\widetilde{l}$ does not bound a flat half plane.
On the contrary, we assume that $\widetilde{l}$ bounds a flat half plane $E$.
By $\mathrm{Pr}:\widetilde{K}_{\Gamma} \to K_{\Gamma}$, the unit semicircle $C\subset E$ centered at $\widetilde{o}$ is isometric to the path $P$ of length $\pi$ from $v_1^{+}$ to $\alpha_{2,n}^{-}$ in $L_{\Gamma}$ (see the bold semicircle in Figure~\ref{l_2_3_fig}).
It follows that $C$ goes through lifts of $v_1^+$, $v_2^{-}$, $x_2^{-}$ and $\alpha_{2,n}^-$.
Therefore, $E$ contains $2$-cells corresponding to relations $x_1=v_1v_2$, $x_2=v_2v_3$ and $x_2=\alpha_{2,n}v_2$ around $\widetilde{o}$.
The unit semicircle in $E$ centered at $v_1^{-1}(\widetilde{o})$
should go through lifts of $v_1^{-}$, $x_1^{-}$ and $v_3^{+}$ (see the dotted semicircle in Figure~\ref{l_2_3_fig}).
This is impossible, since there is no path of length $\pi$ in $L_{\Gamma}$ from $v_1^{-}$ to $v_3^{+}$ through $x_1^{-}$ (see Figure~\ref{link_2_3_fig}).

Since we detected a rank one isometry $\alpha_{2,n}v_3v_1$ in $A_{\Gamma}$,
Theorem~\ref{acyl_thm} shows that $A_{\Gamma}$ is acylindrically hyperbolic.

$(2)$ We discuss the case where $\Gamma$ contains $\Gamma_{2,2,2}$ as a full subgraph.
The argument is similar to $(1)$.
According to Lemma~\ref{pres_gen_lem}, $A_{\Gamma_{2,2,2}}$ has the following presentation:

        \begin{align}
	\left\langle 
	\begin{array}{l|l}
	v_1, v_2, v_3, v_4 &x_1=v_1v_2, x_1=v_2v_1,\\
	x_1, x_2, x_3 &x_2=v_2v_3, x_2=v_3v_2,\\
	 		      &x_3=v_3v_4, x_3=v_4v_3
	\end{array}
	\right\rangle. \label{pres_2_2_2_eq}
	\end{align}

Let $K_{\Gamma_{2,2,2}}\subset K_{\Gamma}$ be the presentation complex of (\ref{pres_2_2_2_eq}).
Let $L_{\Gamma_{2,2,2}}\subset L_{\Gamma}$ be the link in $K_{\Gamma_{2,2,2}}$.
We draw $L_{\Gamma_{2,2,2}}$ on the righthandside of Figure~\ref{link_2_3_fig}.

We show that $x_1x_3$ acts as a rank one isometry on $\widetilde{K}_{\Gamma}$. 
Let $l$ be the concatenation of $1$-cells $x_1$ and $x_3$ in $K_{\Gamma}$.
As in $(1)$, we can find a shortest path $Q$ from $x_3^+$ to $x_1^-$ uniquely in $L_\Gamma$ (see the bold line in Figure~\ref{link_2_3_fig}).
This path $Q$ is of length $\pi$.
In particular, we have $d_{L_{\Gamma}}(x_3^{+}, x_1^{-})=\pi$. 
Similarly we confirm that $d_{L_{\Gamma}}(x_1^{+},x_3^{-})= \pi$. 
Therefore $l$ is a local geodesic,
and the lift $\widetilde{l}$ through $\widetilde{o}$ is an axis of $x_1x_3$, see Figure~\ref{l_2_2_2_fig}.
Now assume that $\widetilde{l}$ bounds a flat half plane $E$.
By $\mathrm{Pr}:\widetilde{K}_{\Gamma} \to K_{\Gamma}$, the unit semicircle $C\subset E$ centered at $\widetilde{o}$ is isometric to the path $Q$ of length $\pi$ from $x_3^{+}$ to $x_1^{-}$ in $L_{\Gamma}$
(see the bold semicircle in Figure~\ref{l_2_2_2_fig}).
It follows that $C$ goes through lifts of $x_3^+$, $v_3^{+}$, $v_2^{-}$ and $x_1^-$.
Therefore, $E$ contains lifts of $2$-cells corresponding to relations $x_3=v_4v_3$, $x_2=v_3v_2$ and $x_1=v_2v_1$
around $\widetilde{o}$.
Then the unit semicircle in $E$ centered at $x_1(\widetilde{o})$
should go through lifts of $x_1^{+}$, $v_1^{+}$ and $x_3^{-}$ (see the dotted semicircle in Figure~\ref{l_2_2_2_fig}).
This is impossible, since there is no path of length $\pi$ in $L_{\Gamma}$ from $x_1^{+}$ to $x_3^{-}$ through $v_1^{+}$ (see Figure~\ref{link_2_3_fig}).
Theorem~\ref{acyl_thm} shows that $A_{\Gamma}$ is acylindrically hyperbolic.
	

	\begin{figure}
	\begin{center}
	\hspace{0.3cm}
	\scalebox{0.8}{
{\unitlength 0.1in%
\begin{picture}(38.3000,20.8500)(5.8000,-22.8000)%
\put(14.0000,-2.6500){\makebox(0,0){{\color[named]{Black}{\large $x_1^{+}$}}}}%
\put(26.0000,-2.6500){\makebox(0,0){{\color[named]{Black}{\large $x_2^{+}$}}}}%
\put(26.1000,-23.4500){\makebox(0,0){{\color[named]{Black}{\large $x_2^{-}$}}}}%
\put(14.1000,-23.4500){\makebox(0,0){{\color[named]{Black}{\large $x_1^{-}$}}}}%
\put(8.0000,-5.8000){\makebox(0,0){{\color[named]{Black}{\large $v_1^{+}$}}}}%
\put(8.3000,-9.8000){\makebox(0,0){{\color[named]{Black}{\large \bf $P$}}}}%
\put(20.0000,-5.8500){\makebox(0,0){{\color[named]{Black}{\large $v_2^{+}$}}}}%
\put(32.0000,-5.8500){\makebox(0,0){{\color[named]{Black}{\large $v_3^{+}$}}}}%
\put(32.0000,-19.800){\makebox(0,0){{\color[named]{Black}{\large $v_3^{-}$}}}}%
\put(20.0000,-19.800){\makebox(0,0){{\color[named]{Black}{\large $v_2^{-}$}}}}%
\put(8.0000,-19.800){\makebox(0,0){{\color[named]{Black}{\large $v_1^{-}$}}}}%
%
{\color[named]{Black}{%
\special{pn 8}%
\special{pa 1410 400}%
\special{pa 1410 400}%
\special{fp}%
}}%
%
{\color[named]{Black}{%
\special{pn 8}%
\special{pa 1410 400}%
\special{pa 2010 800}%
\special{fp}%
\special{pa 1410 400}%
\special{pa 810 800}%
\special{fp}%
}}%
%
{\color[named]{Black}{%
\special{pn 8}%
\special{pa 1610 800}%
\special{pa 1410 400}%
\special{fp}%
}}%
%
{\color[named]{Black}{%
\special{pn 8}%
\special{pa 1410 400}%
\special{pa 1210 800}%
\special{fp}%
}}%
%
{\color[named]{Black}{%
\special{pn 8}%
\special{pa 2610 400}%
\special{pa 3210 800}%
\special{fp}%
\special{pa 2610 400}%
\special{pa 2010 800}%
\special{fp}%
}}%
%
{\color[named]{Black}{%
\special{pn 8}%
\special{pa 2810 800}%
\special{pa 2610 400}%
\special{fp}%
}}%
%
{\color[named]{Black}{%
\special{pn 8}%
\special{pa 2610 400}%
\special{pa 2410 800}%
\special{fp}%
}}%
%

%
{\color[named]{Black}{%
\special{pn 8}%
\special{pa 1410 2180}%
\special{pa 1410 2180}%
\special{fp}%
}}%
%
{\color[named]{Black}{%
\special{pn 8}%
\special{pa 1410 2180}%
\special{pa 2010 1780}%
\special{fp}%
\special{pa 1410 2180}%
\special{pa 810 1780}%
\special{fp}%
}}%
%
{\color[named]{Black}{%
\special{pn 8}%
\special{pa 1610 1780}%
\special{pa 1410 2180}%
\special{fp}%
}}%
%
{\color[named]{Black}{%
\special{pn 8}%
\special{pa 1410 2180}%
\special{pa 1210 1780}%
\special{fp}%
}}%
%
{\color[named]{Black}{%
\special{pn 8}%
\special{pa 2610 2180}%
\special{pa 3210 1780}%
\special{fp}%
}}%
{\color[named]{Black}{%
\special{pn 30}%
\special{pa 2610 2180}%
\special{pa 2010 1780}%
\special{fp}%
}}%
%
{\color[named]{Black}{%
\special{pn 8}%
\special{pa 2810 1780}%
\special{pa 2610 2180}%
\special{fp}%
}}%
%
{\color[named]{Black}{%
\special{pn 30}%
\special{pa 2610 2180}%
\special{pa 2410 1780}%
\special{fp}%
}}%
%

%

%

%
{\color[named]{Black}{%
\special{pn 30}%
\special{pa 810 800}%
\special{pa 2010 1780}%
\special{fp}%
}}%
%
{\color[named]{Black}{%
\special{pn 8}%
\special{pa 2000 800}%
\special{pa 1600 1800}%
\special{fp}%
}}%
%
{\color[named]{Black}{%
\special{pn 8}%
\special{pa 1210 800}%
\special{pa 810 1780}%
\special{fp}%
}}%
%
{\color[named]{Black}{%
\special{pn 8}%
\special{pa 2410 800}%
\special{pa 2010 1780}%
\special{fp}%
}}%
%
{\color[named]{Black}{%
\special{pn 8}%
\special{pa 2800 1800}%
\special{pa 3200 800}%
\special{fp}%
}}%
%
{\color[named]{Black}{%
\special{pn 8}%
\special{pa 3210 1780}%
\special{pa 2010 800}%
\special{fp}%
}}%
%

%

%

\put(27.3000,-16.6000){\makebox(0,0){{\color[named]{Black}{\large $\alpha_{2,3}^{-}$}}}}%
\put(23.3000,-16.6000){\makebox(0,0){{\color[named]{Black}{\large $\alpha_{2,n}^{-}$}}}}%
\put(15.3000,-16.6000){\makebox(0,0){{\color[named]{Black}{\large $\alpha_{1,3}^{-}$}}}}%
\put(11.3000,-16.6000){\makebox(0,0){{\color[named]{Black}{\large $\alpha_{1,m}^{-}$}}}}%
\put(13.2000,-8.6000){\makebox(0,0){{\color[named]{Black}{\large $\alpha_{1,m}^{+}$}}}}%
\put(17.2000,-8.6000){\makebox(0,0){{\color[named]{Black}{\large $\alpha_{1,3}^{+}$}}}}%
\put(25.2000,-8.6000){\makebox(0,0){{\color[named]{Black}{\large $\alpha_{2,n}^{+}$}}}}%
\put(29.2000,-8.6000){\makebox(0,0){{\color[named]{Black}{\large $\alpha_{2,3}^{+}$}}}}%
%

%
{\color[named]{Black}{%
\special{pn 4}%
\special{sh 1}%
\special{ar 2530 1260 10 10 0 6.2831853}%
\special{sh 1}%
\special{ar 2730 1260 10 10 0 6.2831853}%
\special{sh 1}%
\special{ar 2620 1260 10 10 0 6.2831853}%
\special{sh 1}%
\special{ar 2620 1260 10 10 0 6.2831853}%
}}%
%
{\color[named]{Black}{%
\special{pn 4}%
\special{sh 1}%
\special{ar 1330 1260 10 10 0 6.2831853}%
\special{sh 1}%
\special{ar 1530 1260 10 10 0 6.2831853}%
\special{sh 1}%
\special{ar 1420 1260 10 10 0 6.2831853}%
\special{sh 1}%
\special{ar 1420 1260 10 10 0 6.2831853}%
}}%
%

%

%

%

{\color[named]{Black}{%
\special{pn 30}%
\special{pa 1410 2180}%
\special{pa 810 1780}%
\special{dt 0.045}%
}}%

\end{picture}}%
	}
	\hspace{-1.7cm}
	\scalebox{0.72}{
{\unitlength 0.1in%
\begin{picture}(29.1000,22.0500)(1.3000,-24.0500)%
%
{\color[named]{Black}{%
\special{pn 8}%
\special{pa 600 782}%
\special{pa 1007 400}%
\special{fp}%
\special{pa 1007 400}%
\special{pa 1413 782}%
\special{fp}%
\special{pa 1413 782}%
\special{pa 1820 400}%
\special{fp}%
\special{pa 1820 400}%
\special{pa 2227 782}%
\special{fp}%
\special{pa 2227 782}%
\special{pa 2633 400}%
\special{fp}%
\special{pa 2633 400}%
\special{pa 3040 782}%
\special{fp}%
}}%

{\color[named]{Black}{%
\special{pn 30}%
\special{pa 1007 400}%
\special{pa 600 782}%
\special{dt 0.045}%
}}%

%
{\color[named]{Black}{%
\special{pn 8}%
\special{pa 600 1928}%
\special{pa 1007 2310}%
\special{fp}%
\special{pa 1007 2310}%
\special{pa 1413 1928}%
\special{fp}%
\special{pa 1413 1928}%
\special{pa 1820 2310}%
\special{fp}%
\special{pa 1820 2310}%
\special{pa 2227 1928}%
\special{fp}%
\special{pa 2227 1928}%
\special{pa 2633 2310}%
\special{fp}%
\special{pa 2633 2310}%
\special{pa 3040 1928}%
\special{fp}%
}}%
%
{\color[named]{Black}{%
\special{pn 8}%
\special{pa 600 1928}%
\special{pa 1413 782}%
\special{fp}%
\special{pa 1413 782}%
\special{pa 2227 1928}%
\special{fp}%
\special{pa 2227 1928}%
\special{pa 3040 782}%
\special{fp}%
\special{pa 3040 1928}%
\special{pa 2227 782}%
\special{fp}%
\special{pa 2227 782}%
\special{pa 1413 1928}%
\special{fp}%
\special{pa 1413 1928}%
\special{pa 600 782}%
\special{fp}%
}}%
\put(10.0000,-2.6500){\makebox(0,0){{\color[named]{Black}{\large $x_{1}^{+}$}}}}%
\put(18.0000,-2.6500){\makebox(0,0){{\color[named]{Black}{\large $x_2^{+}$}}}}%
\put(26.3000,-2.6500){\makebox(0,0){{\color[named]{Black}{\large $x_3^{+}$}}}}%
\put(25.3000,-7.00){\makebox(0,0){{\color[named]{Black}{\Large \bf $Q$}}}}%
\put(26.3000,-24.7000){\makebox(0,0){{\color[named]{Black}{\large $x_3^{-}$}}}}%
\put(18.2000,-24.7000){\makebox(0,0){{\color[named]{Black}{\large $x_2^{-}$}}}}%
\put(10.1000,-24.7000){\makebox(0,0){{\color[named]{Black}{\large $x_1^{-}$}}}}%
\put(4.1000,-7.8000){\makebox(0,0){{\color[named]{Black}{\large $v_1^{+}$}}}}%
\put(14.1000,-6.000){\makebox(0,0){{\color[named]{Black}{\large $v_2^{+}$}}}}%
\put(22.1000,-6.000){\makebox(0,0){{\color[named]{Black}{\large $v_3^{+}$}}}}%
\put(31.7000,-7.8000){\makebox(0,0){{\color[named]{Black}{\large $v_4^{+}$}}}}%
\put(31.7000,-19.2000){\makebox(0,0){{\color[named]{Black}{\large $v_4^{-}$}}}}%
\put(4.4000,-19.2000){\makebox(0,0){{\color[named]{Black}{\large $v_1^{-}$}}}}%
\put(14.1000,-20.8000){\makebox(0,0){{\color[named]{Black}{\large $v_2^{-}$}}}}%
\put(22.2000,-20.8000){\makebox(0,0){{\color[named]{Black}{\large $v_3^{-}$}}}}%
%
{\color[named]{Black}{%
\special{pn 4}%
\special{sh 1}%
\special{ar 1000 2315 16 16 0 6.2831853}%
\special{sh 1}%
\special{ar 600 1915 16 16 0 6.2831853}%
\special{sh 1}%
\special{ar 1400 1915 16 16 0 6.2831853}%
\special{sh 1}%
\special{ar 2220 1915 16 16 0 6.2831853}%
\special{sh 1}%
\special{ar 1820 2315 16 16 0 6.2831853}%
\special{sh 1}%
\special{ar 2620 2315 16 16 0 6.2831853}%
\special{sh 1}%
\special{ar 3040 1915 16 16 0 6.2831853}%
\special{sh 1}%
\special{ar 3040 795 16 16 0 6.2831853}%
\special{sh 1}%
\special{ar 2220 795 16 16 0 6.2831853}%
\special{sh 1}%
\special{ar 1410 795 16 16 0 6.2831853}%
\special{sh 1}%
\special{ar 600 795 16 16 0 6.2831853}%
\special{sh 1}%
\special{ar 1000 395 16 16 0 6.2831853}%
\special{sh 1}%
\special{ar 1810 395 16 16 0 6.2831853}%
\special{sh 1}%
\special{ar 2630 395 16 16 0 6.2831853}%
\special{sh 1}%
\special{ar 2630 395 16 16 0 6.2831853}%
}}%
%
%
{\color[named]{Black}{%
\special{pn 30}%
\special{pa 2620 405}%
\special{pa 2220 805}%
\special{fp}%
\special{pa 2220 805}%
\special{pa 1420 1925}%
\special{fp}%
\special{pa 1420 1925}%
\special{pa 1020 2295}%
\special{fp}%
}}%
\end{picture}}%
	}
	\caption{$L_{\Gamma_{m,n}}$ and $L_{\Gamma_{2,2,2}}$.}\label{link_2_3_fig}
	\end{center}
	\end{figure}
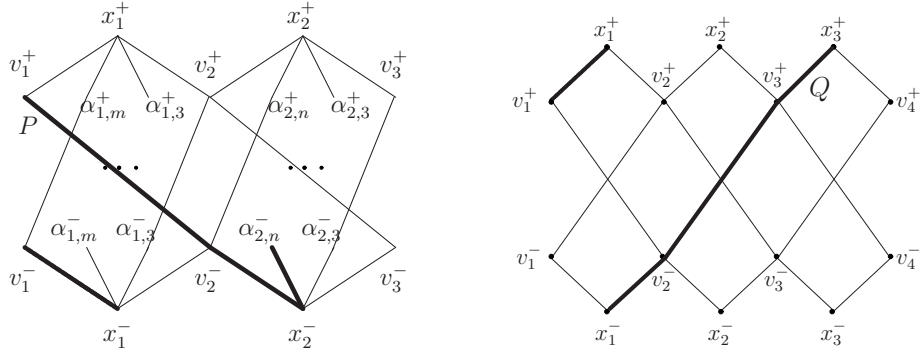
	
	\begin{figure}
	\begin{center}
	\scalebox{0.9}{
	\input{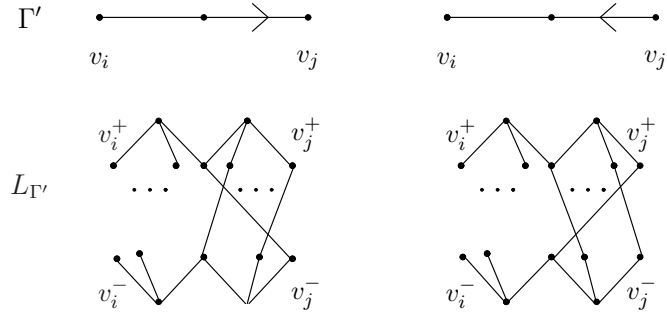}
	}
	\caption{Assuming that $v_i$ and $v_j$ are not adjacent in $\Gamma$, the distance between different points of $\{v_i^{\pm}, v_j^{\pm}\}$ in $L_{\Gamma}$ is at least $\pi$. The minimum $\pi$ occurs only in cases where $v_i$ and $v_j$ are connected by a $2$-path $\Gamma'$ in $\Gamma$.}\label{2links_fig}
	\end{center}
	\end{figure}
	
	\begin{figure}
	\begin{center}
	\scalebox{0.8}{
{\unitlength 0.1in%
\begin{picture}(32.0000,11.9800)(6.2300,-14.7500)%
\put(17.8000,-15.4000){\makebox(0,0){{\color[named]{Black}{\Large ${v}_1$}}}}%
\put(26.9000,-15.4000){\makebox(0,0){{\color[named]{Black}{\Large ${\alpha}_{2,n}$}}}}%
\put(22.3500,-15.4000){\makebox(0,0){{\color[named]{Black}{\Large $\widetilde{o}$}}}}%

\put(16.3900,-7.7500){\makebox(0,0){{\color[named]{Black}{\Large ${x}_1$ }}}}%
\put(20.3900,-9.2000){\makebox(0,0){{\color[named]{Black}{\Large ${v}_2$}}}}%
\put(26.4700,-3.200){\makebox(0,0){{\color[named]{Black}{\Large ${v}_3$}}}}%
\put(25.6100,-8.0500){\makebox(0,0){{\color[named]{Black}{\Large ${x}_2$}}}}%
\put(33.0300,-9.3400){\makebox(0,0){{\color[named]{Black}{\Large ${v}_2$}}}}%
%
{\color[named]{Black}{%
\special{pn 8}%
\special{pa 3070 470}%
\special{pa 3070 1360}%
\special{fp}%
\special{pa 3070 1360}%
\special{pa 3070 1360}%
\special{fp}%
}}%
%
{\color[named]{Black}{%
\special{pn 8}%
\special{pa 2220 1360}%
\special{pa 3070 470}%
\special{fp}%
}}%
%
{\color[named]{Black}{%
\special{pn 8}%
\special{pa 3070 470}%
\special{pa 2220 470}%
\special{fp}%
\special{pa 2220 470}%
\special{pa 2220 1360}%
\special{fp}%
}}%
%
{\color[named]{Black}{%
\special{pn 8}%
\special{pa 1360 1370}%
\special{pa 2220 470}%
\special{fp}%
}}%
%
{\color[named]{Black}{%
\special{pn 30}%
\special{ar 2223 1374 260 260 3.1415927 6.2831853}%
}}%
%
{\color[named]{Black}{%
\special{pn 30}%
\special{pa 2223 1374}%
\special{pa 3823 1374}%
\special{fp}%
\special{pa 2223 1374}%
\special{pa 623 1374}%
\special{fp}%
}}%
%
{\color[named]{Black}{%
\special{pn 8}%
\special{pa 1023 1374}%
\special{pa 920 1310}%
\special{fp}%
\special{pa 920 1438}%
\special{pa 1023 1374}%
\special{fp}%
}}%
%
{\color[named]{Black}{%
\special{pn 8}%
\special{pa 1890 1370}%
\special{pa 1787 1306}%
\special{fp}%
\special{pa 1787 1434}%
\special{pa 1890 1370}%
\special{fp}%
}}%
%
{\color[named]{Black}{%
\special{pn 8}%
\special{pa 2740 1370}%
\special{pa 2637 1306}%
\special{fp}%
\special{pa 2637 1434}%
\special{pa 2740 1370}%
\special{fp}%
}}%
%
{\color[named]{Black}{%
\special{pn 8}%
\special{pa 3490 1370}%
\special{pa 3387 1306}%
\special{fp}%
\special{pa 3387 1434}%
\special{pa 3490 1370}%
\special{fp}%
}}%
\put(5.000,-13.4000){\makebox(0,0){{\color[named]{Black}{\Large $\widetilde{l}$}}}}%
\put(10.9000,-15.4000){\makebox(0,0){{\color[named]{Black}{\Large $v_3$}}}}%
\put(34.9000,-15.4000){\makebox(0,0){{\color[named]{Black}{\Large ${v}_3$}}}}%
%
%
{\color[named]{Black}{%
\special{pn 8}%
\special{pa 2690 470}%
\special{pa 2587 406}%
\special{fp}%
\special{pa 2587 534}%
\special{pa 2690 470}%
\special{fp}%
}}%
%
{\color[named]{Black}{%
\special{pn 8}%
\special{pa 3070 910}%
\special{pa 3000 1020}%
\special{fp}%
\special{pa 3140 1020}%
\special{pa 3070 910}%
\special{fp}%
}}%
%
{\color[named]{Black}{%
\special{pn 8}%
\special{pa 2220 910}%
\special{pa 2150 1020}%
\special{fp}%
\special{pa 2290 1020}%
\special{pa 2220 910}%
\special{fp}%
}}%
%
{\color[named]{Black}{%
\special{pn 8}%
\special{pa 2637 921}%
\special{pa 2516 969}%
\special{fp}%
\special{pa 2629 1051}%
\special{pa 2637 921}%
\special{fp}%
}}%
%
{\color[named]{Black}{%
\special{pn 8}%
\special{pa 1797 911}%
\special{pa 1676 959}%
\special{fp}%
\special{pa 1789 1041}%
\special{pa 1797 911}%
\special{fp}%
}}%
%
{\color[named]{Black}{%
\special{pn 30}%
\special{pn 30}%
\special{pa 1110 1370}%
\special{pa 1111 1350}%
\special{fp}%
\special{pa 1114 1326}%
\special{pa 1118 1307}%
\special{fp}%
\special{pa 1125 1283}%
\special{pa 1132 1265}%
\special{fp}%
\special{pa 1143 1243}%
\special{pa 1153 1226}%
\special{fp}%
\special{pa 1168 1206}%
\special{pa 1181 1191}%
\special{fp}%
\special{pa 1199 1174}%
\special{pa 1214 1162}%
\special{fp}%
\special{pa 1234 1148}%
\special{pa 1252 1139}%
\special{fp}%
\special{pa 1274 1128}%
\special{pa 1293 1122}%
\special{fp}%
\special{pa 1316 1115}%
\special{pa 1336 1112}%
\special{fp}%
\special{pa 1360 1110}%
\special{pa 1381 1110}%
\special{fp}%
\special{pa 1405 1112}%
\special{pa 1424 1116}%
\special{fp}%
\special{pa 1448 1122}%
\special{pa 1466 1128}%
\special{fp}%
\special{pa 1488 1139}%
\special{pa 1506 1148}%
\special{fp}%
\special{pa 1526 1163}%
\special{pa 1542 1175}%
\special{fp}%
\special{pa 1559 1192}%
\special{pa 1572 1206}%
\special{fp}%
\special{pa 1586 1226}%
\special{pa 1596 1243}%
\special{fp}%
\special{pa 1608 1265}%
\special{pa 1615 1283}%
\special{fp}%
\special{pa 1622 1307}%
\special{pa 1626 1326}%
\special{fp}%
\special{pa 1629 1350}%
\special{pa 1630 1370}%
\special{fp}%
}}%
\end{picture}}%
	}
	\caption{An axis $\widetilde{l}$ of $\alpha_{2,n}v_3v_1$.}\label{l_2_3_fig}
	\end{center}
	\end{figure}
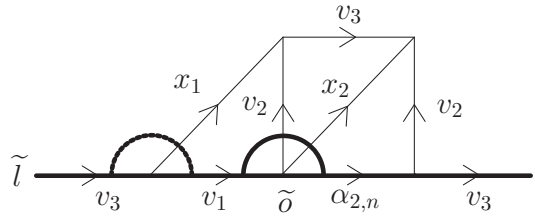

	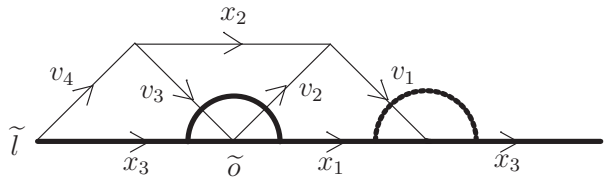
\begin{figure}
	\begin{center}
	\scalebox{0.9}{
{\unitlength 0.1in%
\begin{picture}(32.5000,8.7000)(2.7600,-12.8100)%
%
{\color[named]{Black}{%
\special{pn 8}%
\special{pa 846 634}%
\special{pa 1411 1200}%
\special{fp}%
\special{pa 846 634}%
\special{pa 280 1200}%
\special{fp}%
}}%
%
{\color[named]{Black}{%
\special{pn 30}%
\special{pa 2546 1204}%
\special{pa 286 1204}%
\special{fp}%
}}%
%
{\color[named]{Black}{%
\special{pn 30}%
\special{pa 2526 1204}%
\special{pa 3526 1204}%
\special{fp}%
}}%
%
{\color[named]{Black}{%
\special{pn 8}%
\special{pa 1970 640}%
\special{pa 850 640}%
\special{fp}%
}}%
%
{\color[named]{Black}{%
\special{pn 8}%
\special{pa 1182 962}%
\special{pa 1142 849}%
\special{fp}%
\special{pa 1142 849}%
\special{pa 1142 849}%
\special{fp}%
}}%
%
{\color[named]{Black}{%
\special{pn 8}%
\special{pa 1066 936}%
\special{pa 1182 962}%
\special{fp}%
}}%
%
{\color[named]{Black}{%
\special{pn 8}%
\special{pa 2054 1208}%
\special{pa 1950 1150}%
\special{fp}%
\special{pa 1950 1150}%
\special{pa 1950 1150}%
\special{fp}%
}}%
%
{\color[named]{Black}{%
\special{pn 8}%
\special{pa 1950 1266}%
\special{pa 2054 1208}%
\special{fp}%
}}%
%
{\color[named]{Black}{%
\special{pn 8}%
\special{pa 3054 1208}%
\special{pa 2950 1150}%
\special{fp}%
\special{pa 2950 1150}%
\special{pa 2950 1150}%
\special{fp}%
}}%
%
{\color[named]{Black}{%
\special{pn 8}%
\special{pa 2950 1266}%
\special{pa 3054 1208}%
\special{fp}%
}}%
%
{\color[named]{Black}{%
\special{pn 8}%
\special{pa 924 1208}%
\special{pa 820 1150}%
\special{fp}%
\special{pa 820 1150}%
\special{pa 820 1150}%
\special{fp}%
}}%
%
{\color[named]{Black}{%
\special{pn 8}%
\special{pa 820 1266}%
\special{pa 924 1208}%
\special{fp}%
}}%
%
{\color[named]{Black}{%
\special{pn 8}%
\special{pa 1460 640}%
\special{pa 1356 572}%
\special{fp}%
\special{pa 1356 572}%
\special{pa 1356 572}%
\special{fp}%
}}%
%
{\color[named]{Black}{%
\special{pn 8}%
\special{pa 1360 696}%
\special{pa 1464 638}%
\special{fp}%
}}%
%
{\color[named]{Black}{%
\special{pn 8}%
\special{pa 605 881}%
\special{pa 490 910}%
\special{fp}%
\special{pa 490 910}%
\special{pa 490 910}%
\special{fp}%
}}%
%
{\color[named]{Black}{%
\special{pn 8}%
\special{pa 570 994}%
\special{pa 605 881}%
\special{fp}%
}}%
%
{\color[named]{Black}{%
\special{pn 8}%
\special{pa 1715 881}%
\special{pa 1600 910}%
\special{fp}%
\special{pa 1600 910}%
\special{pa 1600 910}%
\special{fp}%
}}%
%
{\color[named]{Black}{%
\special{pn 8}%
\special{pa 1680 994}%
\special{pa 1715 881}%
\special{fp}%
}}%
%
{\color[named]{Black}{%
\special{pn 8}%
\special{pa 1970 640}%
\special{pa 2535 1206}%
\special{fp}%
\special{pa 1970 640}%
\special{pa 1404 1206}%
\special{fp}%
}}%
%
{\color[named]{Black}{%
\special{pn 8}%
\special{pa 2302 962}%
\special{pa 2262 849}%
\special{fp}%
\special{pa 2262 849}%
\special{pa 2262 849}%
\special{fp}%
}}%
%
{\color[named]{Black}{%
\special{pn 8}%
\special{pa 2186 936}%
\special{pa 2302 962}%
\special{fp}%
}}%
%
{\color[named]{Black}{%
\special{pn 30}%
\special{ar 1416 1206 265 265 3.1415927 6.2831853}%
}}%
\put(1.5600,-12.00){\makebox(0,0){{\color[named]{Black}{\large $\widetilde{l}$}}}}%
\put(8.5600,-13.4600){\makebox(0,0){{\color[named]{Black}{\large $x_3$}}}}%
\put(14.1600,-13.4600){\makebox(0,0){{\color[named]{Black}{\large $\widetilde{o}$}}}}%
\put(19.7600,-13.4600){\makebox(0,0){{\color[named]{Black}{\large $x_1$}}}}%
\put(4.2600,-8.2600){\makebox(0,0){{\color[named]{Black}{\large $v_4$}}}}%
\put(14.1600,-4.7600){\makebox(0,0){{\color[named]{Black}{\large $x_2$}}}}%
\put(9.4600,-9.2600){\makebox(0,0){{\color[named]{Black}{\large $v_3$}}}}%
\put(18.6000,-9.2000){\makebox(0,0){{\color[named]{Black}{\large $v_2$}}}}%
\put(23.9000,-8.2000){\makebox(0,0){{\color[named]{Black}{\large $v_1$}}}}%
\put(29.9000,-13.3000){\makebox(0,0){{\color[named]{Black}{\large $x_3$}}}}%
%

{\color[named]{Black}{%
\special{pn 30}%
\special{pa 2520 1200}%
\special{pa 2520 1200}%
\special{dt 0.045}%
}}%
%
{\color[named]{Black}{%
\special{pn 30}%
\special{pn 30}%
\special{pa 2230 1200}%
\special{pa 2231 1181}%
\special{fp}%
\special{pa 2233 1157}%
\special{pa 2237 1138}%
\special{fp}%
\special{pa 2243 1115}%
\special{pa 2249 1098}%
\special{fp}%
\special{pa 2258 1076}%
\special{pa 2267 1059}%
\special{fp}%
\special{pa 2279 1038}%
\special{pa 2289 1024}%
\special{fp}%
\special{pa 2305 1006}%
\special{pa 2317 992}%
\special{fp}%
\special{pa 2335 977}%
\special{pa 2349 966}%
\special{fp}%
\special{pa 2368 953}%
\special{pa 2385 943}%
\special{fp}%
\special{pa 2406 933}%
\special{pa 2424 926}%
\special{fp}%
\special{pa 2447 919}%
\special{pa 2465 915}%
\special{fp}%
\special{pa 2489 912}%
\special{pa 2507 910}%
\special{fp}%
\special{pa 2532 910}%
\special{pa 2551 912}%
\special{fp}%
\special{pa 2575 915}%
\special{pa 2593 919}%
\special{fp}%
\special{pa 2616 926}%
\special{pa 2633 933}%
\special{fp}%
\special{pa 2655 943}%
\special{pa 2671 953}%
\special{fp}%
\special{pa 2691 966}%
\special{pa 2706 977}%
\special{fp}%
\special{pa 2722 992}%
\special{pa 2736 1006}%
\special{fp}%
\special{pa 2750 1023}%
\special{pa 2761 1039}%
\special{fp}%
\special{pa 2773 1058}%
\special{pa 2782 1076}%
\special{fp}%
\special{pa 2791 1098}%
\special{pa 2797 1115}%
\special{fp}%
\special{pa 2803 1138}%
\special{pa 2807 1157}%
\special{fp}%
\special{pa 2809 1181}%
\special{pa 2810 1200}%
\special{fp}%
}}%
\end{picture}}%
	}
	\caption{An axis $\widetilde{l}$ of $x_1x_3$.}\label{l_2_2_2_fig}
	\end{center}
	\end{figure}
	
\end{proof}

\begin{lemma}\label{key2_lem}
If $\Gamma$ is a connected triangle-free graph with more than one vertex, then
either $\Gamma$ is a complete bipartite graph
or contains the $3$-path graph as a full subgraph.
\end{lemma}

\begin{proof}
Let $\Gamma$ be a connected triangle-free graph with more than one vertex.
If $\Gamma$ contains an $n$-cycle of $n\geq 5$ as a full subgraph, then $\Gamma$ contains a $3$-path subgraph of the $n$-cycle as a full subgraph.
Otherwise $\Gamma$ does not have odd cycles and thus is a bipartite graph. 
We divide the vertex set of $\Gamma$ into two non-empty subsets $V_W$ and $V_B$ such that every edge connects a vertex in $V_W$ and one in $V_B$.
If $\Gamma$ is not a complete bipartite graph,
then there exist vertices $v_W\in V_W$ and $v_B\in V_B$ 
of graph distance greater than $2$, and
every shortest path from $v_{W}$ to $v_{B}$ is a full subgraph.
For such a path, any $3$-subpath is a full subgraph of $\Gamma$.
An example of $\Gamma$ is shown in Figure~\ref{complete_bipartite_ex_fig},
where the division of the vertex set is drawn as white/black coloring of vertices. 
 
	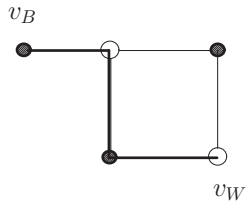
\begin{figure}
	\begin{center}
	\scalebox{0.7}{
{\unitlength 0.1in%
\begin{picture}(16.6100,11.8500)(8.4000,-16.4500)%
\put(9.9000,-5.2500){\makebox(0,0){\Large $v_B$}}%
\put(24.0000,-19.00){\makebox(0,0)[lb]{\Large $v_W$}}%
%
\special{pn 8}%
\special{pa 1630 785}%
\special{pa 1030 785}%
\special{fp}%
%
\special{pn 20}%
\special{pa 2430 1585}%
\special{pa 1630 1585}%
\special{fp}%
\special{pa 1630 1585}%
\special{pa 1630 785}%
\special{fp}%
\special{pa 1630 785}%
\special{pa 1030 785}%
\special{fp}%
%
\special{pn 20}%
\special{ar 2434 780 50 45 0.0000000 6.2831853}%
%
\special{pn 4}%
\special{pa 2472 755}%
\special{pa 2409 812}%
\special{fp}%
\special{pa 2476 764}%
\special{pa 2416 818}%
\special{fp}%
\special{pa 2482 773}%
\special{pa 2429 821}%
\special{fp}%
\special{pa 2479 789}%
\special{pa 2442 823}%
\special{fp}%
\special{pa 2464 749}%
\special{pa 2402 805}%
\special{fp}%
\special{pa 2456 742}%
\special{pa 2394 798}%
\special{fp}%
\special{pa 2444 740}%
\special{pa 2389 789}%
\special{fp}%
\special{pa 2432 737}%
\special{pa 2386 778}%
\special{fp}%
%
\special{pn 8}%
\special{ar 1637 780 67 65 0.0000000 6.2831853}%
%
\special{pn 20}%
\special{ar 1634 1580 50 45 0.0000000 6.2831853}%
%
\special{pn 4}%
\special{pa 1672 1555}%
\special{pa 1609 1612}%
\special{fp}%
\special{pa 1676 1564}%
\special{pa 1616 1618}%
\special{fp}%
\special{pa 1682 1573}%
\special{pa 1629 1621}%
\special{fp}%
\special{pa 1679 1589}%
\special{pa 1642 1623}%
\special{fp}%
\special{pa 1664 1549}%
\special{pa 1602 1605}%
\special{fp}%
\special{pa 1656 1542}%
\special{pa 1594 1598}%
\special{fp}%
\special{pa 1644 1540}%
\special{pa 1589 1589}%
\special{fp}%
\special{pa 1632 1537}%
\special{pa 1586 1578}%
\special{fp}%
%
\special{pn 8}%
\special{ar 2434 1575 67 65 0.0000000 6.2831853}%
%
\special{pn 8}%
\special{pa 1644 1535}%
\special{pa 1644 845}%
\special{fp}%
%
\special{pn 8}%
\special{pa 1704 785}%
\special{pa 2394 785}%
\special{fp}%
%
\special{pn 8}%
\special{pa 2434 785}%
\special{pa 2434 1515}%
\special{fp}%
%
\special{pn 8}%
\special{pa 2374 1585}%
\special{pa 1664 1585}%
\special{fp}%
%
\special{pn 20}%
\special{ar 1000 780 50 45 0.0000000 6.2831853}%
%
\special{pn 4}%
\special{pa 1038 755}%
\special{pa 975 812}%
\special{fp}%
\special{pa 1042 764}%
\special{pa 982 818}%
\special{fp}%
\special{pa 1048 773}%
\special{pa 995 821}%
\special{fp}%
\special{pa 1045 789}%
\special{pa 1008 823}%
\special{fp}%
\special{pa 1030 749}%
\special{pa 968 805}%
\special{fp}%
\special{pa 1022 742}%
\special{pa 960 798}%
\special{fp}%
\special{pa 1010 740}%
\special{pa 955 789}%
\special{fp}%
\special{pa 998 737}%
\special{pa 952 778}%
\special{fp}%
\end{picture}}%
	}
	\end{center}
	\caption{A bipartite graph $\Gamma$ which is not complete.}\label{complete_bipartite_ex_fig}
	\end{figure}
\end{proof}

\begin{proof}[Proof of Theorem~\ref{triangle-free_thm}]
Let $\Gamma$ be a finite simple graph with edges labeled by integers greater than $1$.
Suppose that $\Gamma$ is triangle-free with at least three vertices.

[(i) $\Rightarrow$ (ii)] Suppose that $A_{\Gamma}$ is acylindrically hyperbolic. In general, an acylindrically hyperbolic group does not decompose as a direct product of two infinite groups (\cite{MR3430352}).
Since $K_{\Gamma}$ in Theorem~\ref{CAT(0)_thm} is a finite $K(A_{\Gamma}, 1)$ space,
$A_{\Gamma}$ is torsion free. 
It follows that $A_{\Gamma}$ does not decompose as a direct product of two nontrivial subgroups.

[(ii) $\Rightarrow$ (iii)]  We prove the contrapositive.
If $\Gamma$ is a join of two non-empty subgraphs $\Gamma_1$ and $\Gamma_2$ such that all edges between them are labeled by $2$,
then $A_{\Gamma}$ is the direct product of $A_{\Gamma_1}$ and $A_{\Gamma_2}$.

[(iii) $\Rightarrow$ (iv)] We prove the contrapositive.
If $\Gamma$ is a complete bipartite graph with all edges labeled by $2$, 
then $\Gamma$ is a join of two non-empty subgraphs without any edge $\Gamma_1$ and $\Gamma_2$
such that all edges between them are labeled by $2$.
Therefore, $A_{\Gamma}$ is reducible.

[(iv) $\Rightarrow$ (v)] 
Suppose that $\Gamma$ is not a complete bipartite graph with all edges labeled by $2$.
Under this assumption,
it is enough to consider the following three cases.
$(1)$ $\Gamma$ is not connected,
$(2)$ $\Gamma$ is connected and at least one edge of $\Gamma$ is labeled by an integer greater than $2$, and
$(3)$ $\Gamma$ is connected and all edges of $\Gamma$ are labeled by $2$.
In the second case, $\Gamma$ contains a $2$-path full subgraph with an edge labeled by an integer greater than $2$.
In the third case, 
we can apply Lemma~\ref{key2_lem}, since $\Gamma$ is a connected triangle-free graph with more than one vertex, 
It follows that $\Gamma$ either is a complete bipartite graph or contains a $3$-path full subgraph.
By the assumption, $\Gamma$ contains a $3$-path full subgraph with all edges labeled by $2$.

[(v) $\Rightarrow$ (i)]
First, if $\Gamma$ is disconnected, then $A_{\Gamma}$ decomposes as a free product of two infinite subgroups, and thus acylindrically hyperbolic.
Second, suppose that $\Gamma$ contains a $2$-path full subgraph with an edge labeled by an integer greater than $2$.
Let us assign a direction to the subgraph in the same way as $\Gamma_{m,n}$ in Figure~\ref{key_fig}.
We direct other edges arbitrarily.
With this direction, $\Gamma$ satisfies the assumptions in Lemma~\ref{key1_lem}.
Lemma~\ref{key1_lem} shows that $A_{\Gamma}$ is acylindrically hyperbolic.
Finally, suppose that $\Gamma$ contains a $3$-path full subgraph with all edges labeled by $2$.
Let us assign a direction to the subgraph in the same way as $\Gamma_{2,2,2}$ in Figure~\ref{key_fig}.
Other edges are directed arbitrarily.
With this direction, $\Gamma$ satisfies the assumptions in Lemma~\ref{key1_lem}.
Lemma~\ref{key1_lem} shows that $A_{\Gamma}$ is acylindrically hyperbolic.

Suppose that $A_\Gamma$ is acylindrically hyperbolic.
Since an acylindrically hyperbolic group does not admit an infinite center (\cite{MR3430352}), the center of $A_{\Gamma}$ is finite.
Since $A_{\Gamma}$ has a finite $K(A_{\Gamma}, 1)$ space $K_{\Gamma}$,
$A_{\Gamma}$ is torsion free.
It follows that the center of $A_{\Gamma}$ is trivial.

Let $A_{\Gamma}$ be an irreducible triangle-free Artin-Tits group.
When $\Gamma$ has three or more vertices, 
the central quotient of $A_{\Gamma}$ is $A_{\Gamma}$ itself, which is acylindrically hyperbolic.
When $\Gamma$ has less than three vertices, see Remark~\ref{trivial_rem}.
\end{proof}


\subsection{Artin-Tits groups associated to cones over square-free bipartite graphs}\label{triangle_sec}
In this section, we prove Theorem~\ref{generalizations_thm} and Corollary~\ref{square-free_thm}.

\begin{lemma}\label{generalizations_lem}
Let $A_{\Gamma}$ be an Artin-Tits group of almost large type.
Suppose that $\Gamma$ is a complete bipartite graph with all edges labeled by $2$.
Then $\Gamma$ is a cone over a graph consisting of only isolated vertices with all edges labeled by $2$.
\end{lemma}
\begin{proof}
Assume that $\Gamma$ is not a cone.
Then $\Gamma$ has a square with all edges labeled by $2$,
which contradicts the assumption that $A_{\Gamma}$ is of almost large type. 
\end{proof}

\begin{proof}[Proof of Theorem~\ref{generalizations_thm}]
[(i) $\Rightarrow$ (ii) $\Rightarrow$ (iii) $\Rightarrow$ (iv)]
We repeat the same argument as in the proof of Theorem~\ref{triangle-free_thm}. 

[(iv) $\Rightarrow$ (i)]
Let $A_{\Gamma}$ be triangle-free. 
By Lemma~\ref{generalizations_lem}, (iv) in Theorem~\ref{generalizations_thm} implies (iv) in Theorem~\ref{triangle-free_thm}.
Therefore, according to Theorem~\ref{triangle-free_thm}, $A_{\Gamma}$ is acylindrically hyperbolic.

We consider the case where $\Gamma$ has a $3$-cycle. 
Let us assign $\Gamma$ an appropriate direction.
Let $\widetilde{K}_{\Gamma}$ be the $A_{\Gamma}$-equivariant CAT(0) space in Theorem~\ref{generalizations_CAT(0)_thm}.
Let $L_{\Gamma}$ be the link of the unique vertex $o$ of $K_{\Gamma}$.
We find a rank one isometry in $A_{\Gamma}$.
We fix a $3$-cycle of $\Gamma$.
Then, as a subgraph of the $3$-cycle,
we take a $2$-path subgraph $\Gamma'$ in Figure~\ref{ordered_vertices_fig}.
By Lemma~\ref{pres_gen_lem}, $A_{\Gamma'}$ has the following presentation:
        \begin{align}
	\left\langle 
	\begin{array}{l|l}
	v_1, v_2, v_3, x_1, x_2, &x_1=v_1v_2, x_1=v_2\alpha_{1,3}, \ldots, x_1=\alpha_{1,m} v_1,\\
	\alpha_{1,3}, \ldots, \alpha_{1,m},&x_2=v_2v_3, x_2=v_3\alpha_{2,3}, \ldots, x_2=\alpha_{2,n} v_2
\\ \alpha_{2,3}, \ldots, \alpha_{2,n} 	\end{array}
	\right\rangle. \label{pres_m_n_eq_2}
	\end{align}

Since $\Gamma'$ is not a full subgraph of $\Gamma$,
$A_{\Gamma'}$ is not a subgroup of $A_{\Gamma}$.
On the other hand, the presentation complex $K_{\Gamma'}$ of (\ref{pres_m_n_eq_2}) is a subcomplex of $K_{\Gamma}$.
Let $L_{\Gamma'}\subset L_{\Gamma}$ be the link in $K_{\Gamma'}$, see Figure~\ref{link2_fig}.
As shown in Figure~\ref{geod_fig}, we take a directed loop $l$ in $K_{\Gamma'}$,
which intersects with $L_{\Gamma'}$ at two points $l^+$ and $l^-$.
In $L_{\Gamma'}$, the distance between $l^{+}$ and $l^{-}$ is greater than $\pi$, see Figure~\ref{link2_fig}. 
Even in $L_{\Gamma}$, the distance between $l^{+}$ and $l^{-}$ is greater than $\pi$.
Indeed, a shortest path from $l^+$ to $l^-$ in $L_{\Gamma}$ through $L_{\Gamma}-L_{\Gamma'}$
should go out from $L_{\Gamma'}$ at a point in $\{v_i^{\pm}\}_{i=1,2,3}$, passes at least one edge in $L_{\Gamma}-L_{\Gamma'}$, and comes back into $L_{\Gamma'}$ at another point in $\{v_i^{\pm}\}_{i=1,2,3}$.
The length of such a path is greater than $\pi$,
since the distance between $l^{\pm}$ and $\{v_i^{\pm}\}_{i=1,2,3}$ in $L_{\Gamma'}$ is greater than $\pi/3$.
Therefore,
	\begin{align}
	\begin{aligned}	\label{angle_eq}
	d_{L_{\Gamma}} (l^{+}, l^{-}) >\pi.
	\end{aligned}
	\end{align}
As in the proof of Lemma~\ref{key1_lem}, $l$ is a local geodesic around $o$,
and the lift $\widetilde{l}$ through $\widetilde{o}$ is an axis of $\gamma=x_2\alpha_{1,3}$ in $A_{\Gamma}$ (see Figure~\ref{geod_fig}).
Assume that $\widetilde{l}$ bounds a flat half plane $E$.
By $\mathrm{Pr}:\widetilde{K}_{\Gamma} \to K_{\Gamma}$, the unit semicircle in $E$ centered at $\widetilde{o}$ on $\widetilde{l}$ is isometric to a path of length $\pi$ from $l^{+}$ to $l^{-}$ in ${L}_{\Gamma}$, contrary to (\ref{angle_eq}). 

When $A_{\Gamma}$ is acylindrically hyperbolic, the same discussion as in the proof of Theorem~\ref{triangle-free_thm} shows that the center of $A_{\Gamma}$ is trivial. 
Similarly, Conjecture~\ref{acyl_conj} can be confirmed.
	
	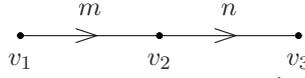
\begin{figure}	
	\begin{center}
	\scalebox{0.9}{
{\unitlength 0.1in%
\begin{picture}(17.5000,2.5500)(2.5000,-4.4000)%
%
{\color[named]{Black}{%
\special{pn 8}%
\special{pa 400 400}%
\special{pa 1200 400}%
\special{fp}%
\special{pa 1200 400}%
\special{pa 1200 400}%
\special{fp}%
}}%
%
{\color[named]{Black}{%
\special{pn 8}%
\special{pa 1200 400}%
\special{pa 2000 400}%
\special{fp}%
}}%
%
{\color[named]{Black}{%
\special{pn 4}%
\special{sh 1}%
\special{ar 1200 400 16 16 0 6.2831853}%
\special{sh 1}%
\special{ar 1200 400 16 16 0 6.2831853}%
}}%
%
{\color[named]{Black}{%
\special{pn 8}%
\special{pa 1640 400}%
\special{pa 1520 360}%
\special{fp}%
\special{pa 1520 440}%
\special{pa 1640 400}%
\special{fp}%
}}%
%
{\color[named]{Black}{%
\special{pn 8}%
\special{pa 840 400}%
\special{pa 720 360}%
\special{fp}%
\special{pa 720 440}%
\special{pa 840 400}%
\special{fp}%
}}%
%
{\color[named]{Black}{%
\special{pn 4}%
\special{sh 1}%
\special{ar 400 400 16 16 0 6.2831853}%
\special{sh 1}%
\special{ar 400 400 16 16 0 6.2831853}%
}}%
%
{\color[named]{Black}{%
\special{pn 4}%
\special{sh 1}%
\special{ar 2000 400 16 16 0 6.2831853}%
\special{sh 1}%
\special{ar 2000 400 16 16 0 6.2831853}%
}}%
\put(4.0000,-5.5000){\makebox(0,0){{\color[named]{Black}{$v_1$}}}}%
\put(8.0000,-2.5000){\makebox(0,0){{\color[named]{Black}{$m$}}}}%
\put(12.0000,-5.5000){\makebox(0,0){{\color[named]{Black}{$v_2$}}}}%
\put(16.0000,-2.5000){\makebox(0,0){{\color[named]{Black}{$n$}}}}%
\put(20.0000,-5.5000){\makebox(0,0){{\color[named]{Black}{$v_3$}}}}%
\end{picture}}%
	}
	\caption{A $2$-path directed graph $\Gamma'$ with edges labeled by $m,n\geq 3$.}\label{ordered_vertices_fig}
	\end{center}
	\end{figure}
	
	\begin{figure}
	\begin{center}
	\hspace{1.8cm}
	\scalebox{0.8}{
	\input{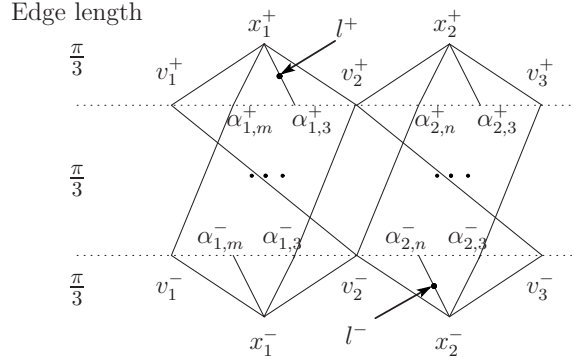}
	}
	\end{center}
	\caption{The subgraph $L_{\Gamma'}$ of $L_{\Gamma}$ corresponding to the $2$-path subgraph $\Gamma'$ of $\Gamma$, and two points ${l}^{\pm}$.}\label{link2_fig}
	\end{figure}
		
	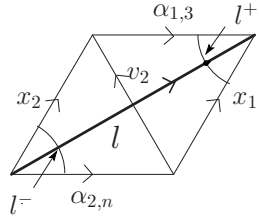
\begin{figure}
	\begin{center}
	\scalebox{0.7}{
{\unitlength 0.1in%
\begin{picture}(20.4500,14.8300)(4.9500,-18.1800)%
%
{\color[named]{Black}{%
\special{pn 8}%
\special{pa 600 1600}%
\special{pa 1800 1600}%
\special{fp}%
}}%
%
{\color[named]{Black}{%
\special{pn 8}%
\special{pa 1800 1600}%
\special{pa 1200 560}%
\special{fp}%
}}%
%
{\color[named]{Black}{%
\special{pn 8}%
\special{pa 1200 560}%
\special{pa 600 1590}%
\special{fp}%
}}%
%
{\color[named]{Black}{%
\special{pn 8}%
\special{pa 1200 1600}%
\special{pa 1140 1530}%
\special{fp}%
}}%
%
{\color[named]{Black}{%
\special{pn 8}%
\special{pa 1200 1600}%
\special{pa 1130 1660}%
\special{fp}%
}}%
%
{\color[named]{Black}{%
\special{pn 8}%
\special{pa 1373 861}%
\special{pa 1363 953}%
\special{fp}%
}}%
%
{\color[named]{Black}{%
\special{pn 8}%
\special{pa 1380 860}%
\special{pa 1471 870}%
\special{fp}%
}}%
%
{\color[named]{Black}{%
\special{pn 8}%
\special{pa 945 1041}%
\special{pa 955 1133}%
\special{fp}%
}}%
%
{\color[named]{Black}{%
\special{pn 8}%
\special{pa 945 1041}%
\special{pa 854 1051}%
\special{fp}%
}}%
\put(12.0000,-18.0000){\makebox(0,0){{\color[named]{Black}{\Large $\alpha_{2,n}$}}}}%
\put(15.4000,-8.6000){\makebox(0,0){{\color[named]{Black}{\Large $v_2$}}}}%
\put(7.3000,-10.5000){\makebox(0,0){{\color[named]{Black}{\Large $x_2$}}}}%
%
{\color[named]{Black}{%
\special{pn 8}%
\special{pa 1200 560}%
\special{pa 2400 560}%
\special{fp}%
}}%
%
{\color[named]{Black}{%
\special{pn 8}%
\special{pa 2400 560}%
\special{pa 1800 1600}%
\special{fp}%
}}%
%
{\color[named]{Black}{%
\special{pn 8}%
\special{pa 1800 570}%
\special{pa 1740 500}%
\special{fp}%
}}%
%
{\color[named]{Black}{%
\special{pn 8}%
\special{pa 1800 570}%
\special{pa 1730 630}%
\special{fp}%
}}%
%
{\color[named]{Black}{%
\special{pn 8}%
\special{pa 2135 1041}%
\special{pa 2044 1051}%
\special{fp}%
}}%
%
{\color[named]{Black}{%
\special{pn 8}%
\special{pa 2135 1041}%
\special{pa 2145 1133}%
\special{fp}%
}}%
\put(23.3000,-10.5000){\makebox(0,0){{\color[named]{Black}{\Large $x_1$}}}}%
\put(18.0000,-4.0000){\makebox(0,0){{\color[named]{Black}{\Large $\alpha_{1,3}$}}}}%
%
{\color[named]{Black}{%
\special{pn 20}%
\special{pa 600 1600}%
\special{pa 2400 560}%
\special{fp}%
}}%
%
{\color[named]{Black}{%
\special{pn 13}%
\special{pa 1780 910}%
\special{pa 1720 880}%
\special{fp}%
}}%
%
{\color[named]{Black}{%
\special{pn 13}%
\special{pa 1810 910}%
\special{pa 1770 990}%
\special{fp}%
}}%
%
{\color[named]{Black}{%
\special{pn 8}%
\special{ar 2380 560 400 400 2.0005586 3.1415927}%
}}%
%
{\color[named]{Black}{%
\special{pn 4}%
\special{sh 1}%
\special{ar 2040 770 16 16 0 6.2831853}%
\special{sh 1}%
\special{ar 2040 770 16 16 0 6.2831853}%
}}%
\put(23.4000,-4.0000){\makebox(0,0){{\color[named]{Black}{\Large $l^{+}$}}}}%
%
{\color[named]{Black}{%
\special{pn 4}%
\special{sh 1}%
\special{ar 2060 760 8 8 0 6.2831853}%
\special{sh 1}%
\special{ar 2060 760 8 8 0 6.2831853}%
}}%
%
{\color[named]{Black}{%
\special{pn 8}%
\special{ar 595 1588 400 400 5.2279380 6.2831853}%
}}%
%
{\color[named]{Black}{%
\special{pn 8}%
\special{pa 675 1818}%
\special{pa 675 1818}%
\special{fp}%
}}%
%
{\color[named]{Black}{%
\special{pn 8}%
\special{pa 2180 500}%
\special{pa 2040 700}%
\special{fp}%
\special{sh 1}%
\special{pa 2040 700}%
\special{pa 2095 657}%
\special{pa 2071 656}%
\special{pa 2062 634}%
\special{pa 2040 700}%
\special{fp}%
}}%
%
{\color[named]{Black}{%
\special{pn 8}%
\special{pa 765 1698}%
\special{pa 935 1428}%
\special{fp}%
\special{sh 1}%
\special{pa 935 1428}%
\special{pa 883 1474}%
\special{pa 907 1473}%
\special{pa 916 1495}%
\special{pa 935 1428}%
\special{fp}%
}}%
\put(6.7500,-18.1800){\makebox(0,0){{\color[named]{Black}{\Large $l^{-}$}}}}%
\put(13.7500,-13.1800){\makebox(0,0){{\color[named]{Black}{\LARGE $l$}}}}%
\end{picture}}%
	}
	\end{center}
	\caption{A directed loop $l$ in $K_{\Gamma'}$ and its intersection points with $L_{\Gamma'}$.}\label{geod_fig}
	\end{figure}
\end{proof}

\begin{proof}[Proof of Corollary~\ref{square-free_thm}]
Let $\Gamma$ satisfy assumptions in Corollary~\ref{square-free_thm}.
By assumptions on edge labels of $\Gamma$, $A_\Gamma$ is of almost large type.
We show that $\Gamma$ can be appropriately directed.
Note that $\Gamma'_2$ is bipartite.
We color vertices of $\Gamma'_2$ by black and white
such that every edge in $\Gamma'_2$ connects a white vertex and a black vertex.
We give a direction to every edge $e$ of $\Gamma$ as follows.
If $e$ is in $\{v_0\}\ast \Gamma'_1$, $e$ is directed arbitrarily.
If $e$ is in $\Gamma'_2$, $e$ goes from a white vertex to a black vertex.
Otherwise $e$ goes from $v_0$ to a white vertex, or goes from a black vertex to $v_0$.
Then every $3$-cycle is the one in Figure~\ref{squares2_fig}, as shown in Figure~\ref{correct_triangle_fig}.
Note that $\Gamma'_2$ has no $4$-cycles.
Since every $4$-cycle shares two edges with such a directed $3$-cycle,
it is the same as the rightmost $4$-cycle in Figure~\ref{squares2_fig}.
Hence $\Gamma$ is appropriately directed.
Since $\Gamma$ satisfies the condition (iv) in Theorem~\ref{generalizations_thm},
$A_{\Gamma}$ satisfies (i), (ii) and (iii) in Theorem~\ref{generalizations_thm}.
	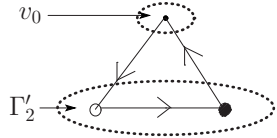
\begin{figure}[h]
	\begin{center}
	\hspace{0cm}
	\scalebox{0.7}{
{\unitlength 0.1in%
\begin{picture}(21.2500,9.8000)(2.3500,-10.3000)%
%
{\color[named]{Black}{%
\special{pn 8}%
\special{ar 1041 845 41 45 0.0000000 6.2831853}%
}}%
%
{\color[named]{Black}{%
\special{pn 8}%
\special{ar 1999 840 42 45 0.0000000 6.2831853}%
}}%
%
{\color[named]{Black}{%
\special{pn 20}%
\special{pa 1999 800}%
\special{pa 1958 845}%
\special{fp}%
\special{pa 1962 863}%
\special{pa 2020 805}%
\special{fp}%
\special{pa 2031 814}%
\special{pa 1970 877}%
\special{fp}%
\special{pa 1981 888}%
\special{pa 2035 829}%
\special{fp}%
\special{pa 1999 888}%
\special{pa 2041 843}%
\special{fp}%
}}%
%
{\color[named]{Black}{%
\special{pn 8}%
\special{sh 1}%
\special{ar 1560 160 16 16 0 6.2831853}%
\special{sh 1}%
\special{ar 1560 160 16 16 0 6.2831853}%
}}%
%
{\color[named]{Black}{%
\special{pn 8}%
\special{pa 1560 160}%
\special{pa 1990 840}%
\special{fp}%
}}%
%
{\color[named]{Black}{%
\special{pn 8}%
\special{pa 1550 160}%
\special{pa 1040 840}%
\special{fp}%
}}%
%
{\color[named]{Black}{%
\special{pn 8}%
\special{pa 1700 370}%
\special{pa 1700 510}%
\special{fp}%
\special{pa 1700 510}%
\special{pa 1700 510}%
\special{fp}%
\special{pa 1700 377}%
\special{pa 1833 443}%
\special{fp}%
}}%
%
{\color[named]{Black}{%
\special{pn 8}%
\special{pa 1200 630}%
\special{pa 1200 490}%
\special{fp}%
\special{pa 1200 490}%
\special{pa 1200 490}%
\special{fp}%
\special{pa 1200 623}%
\special{pa 1333 557}%
\special{fp}%
}}%
%
{\color[named]{Black}{%
\special{pn 8}%
\special{pa 1080 830}%
\special{pa 1970 830}%
\special{fp}%
}}%
%
{\color[named]{Black}{%
\special{pn 8}%
\special{pa 1600 830}%
\special{pa 1500 750}%
\special{fp}%
\special{pa 1500 900}%
\special{pa 1590 830}%
\special{fp}%
}}%
%
{\color[named]{Black}{%
\special{pn 20}%
\special{pn 20}%
\special{pa 1770 160}%
\special{pa 1769 168}%
\special{fp}%
\special{pa 1758 197}%
\special{pa 1753 203}%
\special{fp}%
\special{pa 1728 226}%
\special{pa 1722 230}%
\special{fp}%
\special{pa 1690 246}%
\special{pa 1682 249}%
\special{fp}%
\special{pa 1648 260}%
\special{pa 1640 262}%
\special{fp}%
\special{pa 1605 267}%
\special{pa 1597 268}%
\special{fp}%
\special{pa 1560 270}%
\special{pa 1552 270}%
\special{fp}%
\special{pa 1515 267}%
\special{pa 1507 266}%
\special{fp}%
\special{pa 1471 260}%
\special{pa 1463 258}%
\special{fp}%
\special{pa 1429 246}%
\special{pa 1422 243}%
\special{fp}%
\special{pa 1391 225}%
\special{pa 1385 221}%
\special{fp}%
\special{pa 1362 196}%
\special{pa 1358 191}%
\special{fp}%
\special{pa 1350 159}%
\special{pa 1351 152}%
\special{fp}%
\special{pa 1363 122}%
\special{pa 1367 117}%
\special{fp}%
\special{pa 1392 94}%
\special{pa 1398 90}%
\special{fp}%
\special{pa 1431 73}%
\special{pa 1438 71}%
\special{fp}%
\special{pa 1471 60}%
\special{pa 1479 59}%
\special{fp}%
\special{pa 1515 52}%
\special{pa 1523 52}%
\special{fp}%
\special{pa 1560 50}%
\special{pa 1568 50}%
\special{fp}%
\special{pa 1605 53}%
\special{pa 1613 54}%
\special{fp}%
\special{pa 1648 60}%
\special{pa 1656 62}%
\special{fp}%
\special{pa 1690 73}%
\special{pa 1698 77}%
\special{fp}%
\special{pa 1728 94}%
\special{pa 1734 99}%
\special{fp}%
\special{pa 1758 123}%
\special{pa 1762 129}%
\special{fp}%
\special{pa 1770 160}%
\special{pa 1770 160}%
\special{fp}%
}}%
%
{\color[named]{Black}{%
\special{pn 8}%
\special{pa 1560 830}%
\special{pa 1560 830}%
\special{dt 0.045}%
}}%
%
{\color[named]{Black}{%
\special{pn 20}%
\special{pn 20}%
\special{pa 2360 830}%
\special{pa 2359 838}%
\special{fp}%
\special{pa 2344 870}%
\special{pa 2339 876}%
\special{fp}%
\special{pa 2310 899}%
\special{pa 2303 904}%
\special{fp}%
\special{pa 2271 921}%
\special{pa 2264 925}%
\special{fp}%
\special{pa 2230 939}%
\special{pa 2222 942}%
\special{fp}%
\special{pa 2187 955}%
\special{pa 2179 957}%
\special{fp}%
\special{pa 2143 967}%
\special{pa 2136 969}%
\special{fp}%
\special{pa 2099 978}%
\special{pa 2092 980}%
\special{fp}%
\special{pa 2055 987}%
\special{pa 2047 989}%
\special{fp}%
\special{pa 2011 995}%
\special{pa 2003 997}%
\special{fp}%
\special{pa 1966 1002}%
\special{pa 1958 1004}%
\special{fp}%
\special{pa 1921 1008}%
\special{pa 1913 1010}%
\special{fp}%
\special{pa 1876 1014}%
\special{pa 1868 1014}%
\special{fp}%
\special{pa 1831 1018}%
\special{pa 1823 1019}%
\special{fp}%
\special{pa 1786 1022}%
\special{pa 1778 1022}%
\special{fp}%
\special{pa 1741 1025}%
\special{pa 1733 1025}%
\special{fp}%
\special{pa 1696 1027}%
\special{pa 1688 1027}%
\special{fp}%
\special{pa 1650 1029}%
\special{pa 1642 1029}%
\special{fp}%
\special{pa 1605 1030}%
\special{pa 1597 1030}%
\special{fp}%
\special{pa 1560 1030}%
\special{pa 1551 1030}%
\special{fp}%
\special{pa 1514 1030}%
\special{pa 1506 1030}%
\special{fp}%
\special{pa 1469 1029}%
\special{pa 1461 1029}%
\special{fp}%
\special{pa 1423 1027}%
\special{pa 1415 1027}%
\special{fp}%
\special{pa 1378 1025}%
\special{pa 1370 1024}%
\special{fp}%
\special{pa 1333 1022}%
\special{pa 1325 1021}%
\special{fp}%
\special{pa 1288 1018}%
\special{pa 1280 1017}%
\special{fp}%
\special{pa 1243 1013}%
\special{pa 1235 1013}%
\special{fp}%
\special{pa 1198 1009}%
\special{pa 1190 1007}%
\special{fp}%
\special{pa 1153 1002}%
\special{pa 1145 1001}%
\special{fp}%
\special{pa 1108 995}%
\special{pa 1101 994}%
\special{fp}%
\special{pa 1064 987}%
\special{pa 1056 986}%
\special{fp}%
\special{pa 1020 977}%
\special{pa 1012 975}%
\special{fp}%
\special{pa 976 966}%
\special{pa 968 965}%
\special{fp}%
\special{pa 932 954}%
\special{pa 924 952}%
\special{fp}%
\special{pa 889 939}%
\special{pa 882 936}%
\special{fp}%
\special{pa 848 921}%
\special{pa 841 918}%
\special{fp}%
\special{pa 809 899}%
\special{pa 803 894}%
\special{fp}%
\special{pa 776 869}%
\special{pa 771 863}%
\special{fp}%
\special{pa 760 830}%
\special{pa 761 822}%
\special{fp}%
\special{pa 777 790}%
\special{pa 782 783}%
\special{fp}%
\special{pa 810 760}%
\special{pa 816 756}%
\special{fp}%
\special{pa 849 739}%
\special{pa 856 735}%
\special{fp}%
\special{pa 890 721}%
\special{pa 898 718}%
\special{fp}%
\special{pa 933 705}%
\special{pa 941 703}%
\special{fp}%
\special{pa 977 693}%
\special{pa 985 691}%
\special{fp}%
\special{pa 1021 682}%
\special{pa 1029 680}%
\special{fp}%
\special{pa 1065 673}%
\special{pa 1073 671}%
\special{fp}%
\special{pa 1110 665}%
\special{pa 1118 663}%
\special{fp}%
\special{pa 1154 658}%
\special{pa 1162 656}%
\special{fp}%
\special{pa 1199 652}%
\special{pa 1207 650}%
\special{fp}%
\special{pa 1244 646}%
\special{pa 1252 645}%
\special{fp}%
\special{pa 1289 642}%
\special{pa 1297 641}%
\special{fp}%
\special{pa 1334 638}%
\special{pa 1342 638}%
\special{fp}%
\special{pa 1379 635}%
\special{pa 1387 635}%
\special{fp}%
\special{pa 1425 633}%
\special{pa 1433 633}%
\special{fp}%
\special{pa 1470 631}%
\special{pa 1478 631}%
\special{fp}%
\special{pa 1515 630}%
\special{pa 1523 630}%
\special{fp}%
\special{pa 1561 630}%
\special{pa 1569 630}%
\special{fp}%
\special{pa 1606 630}%
\special{pa 1614 630}%
\special{fp}%
\special{pa 1652 631}%
\special{pa 1660 632}%
\special{fp}%
\special{pa 1697 633}%
\special{pa 1705 633}%
\special{fp}%
\special{pa 1742 635}%
\special{pa 1750 636}%
\special{fp}%
\special{pa 1787 638}%
\special{pa 1795 639}%
\special{fp}%
\special{pa 1832 642}%
\special{pa 1840 643}%
\special{fp}%
\special{pa 1877 647}%
\special{pa 1885 647}%
\special{fp}%
\special{pa 1922 651}%
\special{pa 1930 653}%
\special{fp}%
\special{pa 1967 658}%
\special{pa 1975 659}%
\special{fp}%
\special{pa 2012 665}%
\special{pa 2020 667}%
\special{fp}%
\special{pa 2056 674}%
\special{pa 2064 675}%
\special{fp}%
\special{pa 2100 683}%
\special{pa 2108 684}%
\special{fp}%
\special{pa 2144 694}%
\special{pa 2152 695}%
\special{fp}%
\special{pa 2188 706}%
\special{pa 2196 708}%
\special{fp}%
\special{pa 2231 721}%
\special{pa 2238 724}%
\special{fp}%
\special{pa 2272 739}%
\special{pa 2279 743}%
\special{fp}%
\special{pa 2311 762}%
\special{pa 2317 766}%
\special{fp}%
\special{pa 2344 790}%
\special{pa 2349 797}%
\special{fp}%
}}%
\put(5.6000,-1.4000){\makebox(0,0){{\color[named]{Black}{\Large $v_0$}}}}%
\put(5.000,-8.3000){\makebox(0,0){{\color[named]{Black}{\Large $\Gamma'_2$}}}}%
%
{\color[named]{Black}{%
\special{pn 8}%
\special{pa 630 830}%
\special{pa 890 830}%
\special{fp}%
\special{sh 1}%
\special{pa 890 830}%
\special{pa 823 810}%
\special{pa 837 830}%
\special{pa 823 850}%
\special{pa 890 830}%
\special{fp}%
}}%
%
{\color[named]{Black}{%
\special{pn 8}%
\special{pa 690 140}%
\special{pa 1440 140}%
\special{fp}%
\special{sh 1}%
\special{pa 1440 140}%
\special{pa 1373 120}%
\special{pa 1387 140}%
\special{pa 1373 160}%
\special{pa 1440 140}%
\special{fp}%
}}%
\end{picture}}%
	}
	\caption{A directed $3$-cycle in $\Gamma$.}\label{correct_triangle_fig}
	\end{center}
	\end{figure}
\end{proof}

Let us state one more corollary of Theorem~\ref{generalizations_thm}.
\begin{corollary}\label{square-free_cor}
Let $A_{\Gamma}$ be an Artin-Tits group of almost large type associated to $\Gamma$ with three or more vertices.
Suppose that $\Gamma$ is square-free, that is, $\Gamma$ does not contain $4$-cycles.
When $A_{\Gamma}$ is irreducible, it is acylindrically hyperbolic, directly indecomposable and centerless. 
In particular, Conjecture~\ref{acyl_conj} is true for such an $A_{\Gamma}$. 
\end{corollary}

\begin{proof}
If $\Gamma$ is triangle-free, then we use Theorem~\ref{triangle-free_thm}.
Otherwise, every two triangles of $\Gamma$ either share only a vertex or are disjoint,
since $\Gamma$ is square-free. 
We can assign a direction to all triangles as in Figure~\ref{squares2_fig}.
When other edges of $\Gamma$ are directed arbitrarily, $\Gamma$ is appropriately directed. 
We apply Theorem~\ref{generalizations_thm}.
\end{proof}

\section*{Acknowledgments}

The authors would like to thank the anonymous reviewers for comments on the previous version of this paper.



\end{document}